\documentclass[a4paper,12pt]{amsart}
\usepackage{color, graphicx, wrapfig}
\usepackage[left=1in,top=1in,right=1in]{geometry}

\usepackage{
amsmath,
amsfonts,
latexsym, 
amsrefs,
amssymb,
hyperref,
csquotes
}

\usepackage{bm}

\usepackage{hyperref}

\usepackage{dsfont}

\newcommand{\Tr}{\mbox{Tr\,}}
 
\renewcommand{\Re}{{\rm Re}}

\newcommand{\E}{\mathbb{E}}
\newcommand{\Z}{\mathbb Z}

\newtheorem{thm}{Theorem}

\newtheorem{prop}[thm]{Proposition}

\newtheorem{conj}{Conjecture}

\title{Extrema of log-correlated random variables:\\Principles and Examples}
\author[L.-P. Arguin]{Louis-Pierre ARGUIN}
\thanks{L.-P. A. is supported by the NSF grant DMS 1513441, the PSC-CUNY Research Award~68784-00~46, and partially by a NSERC Discovery grant and FQRNT {\it Nouveaux chercheurs} grant.}
\address{L.-P. Arguin\\
Department of Mathematics, Baruch College and Graduate Center, City University of New York, New York, NY 10010, USA.
}
\email{louis-pierre.arguin@baruch.cuny.edu}


\begin{document}

\begin{abstract}
These notes were written for the mini-course {\it Extrema of log-correlated random variables: Principles and Examples} at the Introductory School held in January 2015 at the Centre International de Rencontres Math\'ematiques in Marseille.
There have been many advances in the understanding of the high values of log-correlated random fields  from the physics and mathematics perspectives in recent years. 
These fields admit correlations that decay approximately like the logarithm of the inverse of the distance between index points.
Examples include branching random walks and the two-dimensional Gaussian free field.
In this paper, we review the properties of such fields and survey the progress in describing the statistics of their extremes.
The branching random walk is used as a guiding example to prove the correct leading and subleading order of the maximum following the multiscale refinement of the second moment method of Kistler. 
The approach sheds light on a conjecture of Fyodorov, Hiary \& Keating on the maximum of the Riemann zeta function on an interval of the critical line and
of the characteristic polynomial of random unitary matrices.

\end{abstract}

\maketitle

\tableofcontents

\section{Introduction}
\label{sect: intro}
There has been tremendous progress recently in the understanding of the extreme value statistics of stochastic processes whose variables exhibit strong correlations. The purpose of this paper is to survey the recent progress in this field for processes with logarithmically decaying correlations. 
The material has been developed for the mini-course {\it Extrema of log-correlated random variables: Principles and Examples} at the Introductory School held in January 2015 at the Centre International de Rencontres Math\'ematiques in Marseille for the trimester {\it Disordered systems, Random spatial processes and Applications}  of the Institut Henri-Poincar\'e.

The study of extreme values of stochastic processes goes back to the early twentieth century. 
The theory was developed in the context of independent or weakly correlated random variables. 
The first book regrouping the early advances in the field is the still relevant {\it Statistics of Extremes} by E.J. Gumbel \cite{gumbel} published in 1958.
Gumbel credits the first work on extreme values to Bortkiewicz in 1898, which is one of the first paper emphasizing the importance of Poisson statistics. 
After progress by Fisher and Tippett in understanding the limit law of the maximum of a collection of IID random variables, the theory culminated in 1943 with the classification theorem of Gnedenko \cite{gnedenko} which showed that the distribution of the maximum of $N$ independent and identically distributed (IID) random variables when properly recentered and rescaled can only be of three types: Fr\'echet, Weibull or Gumbel.
Prior to this, von Mises had shown sufficient conditions for convergence to the three types.
All in all, it took the theory sixty years between the first rigorous results and an essentially complete theory as it appears in Gumbel's review.
There are now many excellent textbooks on extreme value statistics, see e.g. \cite{leadbetter, dehaan-ferreira, resnick,bovier_extremes}.

The classical theory however does not apply to collections of random variables that exhibit strong correlations, i.e.~~ correlations of the order of the variance. As Gumbel points out at the very beginning of {\it Statistics of Extremes}:
\begin{displayquote}
{\it
Another limitation of the theory is the condition that the observations from which the extremes are taken should be independent.
This assumption, made in most statistical work, is hardly ever realized.}
\end{displayquote}
In the last thirty years, there have been important breakthroughs in mathematics and in physics to identify {\it universality classes} 
for the distributions of extreme values of strongly correlated stochastic processes. 
One such class that has attracted much interest in recent years is the class of {\it log-correlated fields}.
Essentially, these are stochastic processes for which the correlations decay logarithmically with the distance. 
Even though the theory has not reached a complete status as in the IID case, there has been tremendous progress in identifying the distribution of extremes in this case and to develop rich heuristics that describe accurately the statistics of large values.
There are already excellent survey papers on the subject by Bovier \cite{bovier_bbm} and Zeitouni \cite{zeitouni_notes},
all offering a different perspective on log-correlated fields. 
The present paper modestly aims at complementing the growing literature on the subject by focusing on the techniques to obtain precise estimates on the order of the maximum of log-correlated fields and their potential applications to seemingly unrelated problems: the maxima of the characteristic polynomials of random matrices and the maxima of the Riemann zeta function on an interval of the critical line.
The approach we take is based on Kistler's multiscale refinement of the second moment method introduced in  \cite{Kistler2014}.


The presentation is organized as follows. 
In Section \ref{sect: intro},  we review the basic theory in the case of IID random variables and will introduce log-correlated fields in generality.
Section \ref{sect: examples} focuses on two important examples of log-correlated fields: {\it branching random walks} and {\it the two-dimensional Gaussian free field}. 
In particular, three properties of these fields (and of log-correlated fields in general) are singled out as they play a crucial role in the analysis of extremes. 
In Section \ref{sect: order}, we describe the general method of Kistler to prove fine asymptotics for the maximum of log-correlated fields \cite{Kistler2014}.
Finally, in Section \ref{sect: universality}, we will explain how the method is expected to be applicable to study the high values of the characteristic polynomials of random matrix ensemble as well as the large values of the Riemann zeta function  on an interval of the critical line.

\medskip
\noindent {\bf Acknowledgements}\\
I am grateful to Jean-Philippe Bouchaud, Pierluigi Contucci, Cristian Giardina, Pierre Nolin, Vincent Vargas, and Vladas Sidoravicius
 for the organization of the trimester {\it Disordered systems, Random spatial processes and Applications}  at the the Institut Henri-Poincar\'e in the spring 2015. I also thank the {\it Centre International de Rencontres Math\'ematiques} and its staff for the hospitality and the organization during the Introductory School. This work would not have been possible without the numerous discussions with my collaborators on the subject: David Belius, Paul Bourgade, Anton Bovier, Adam J. Harper, Nicola Kistler, Olivier Zindy, and my students Samuel April, Fr\'ed\'eric Ouimet, Roberto Persechino, and Jean-S\'ebastien Turcotte.

\subsection{Statistics of extremes}
\label{sect: stat}
We are interested in the problem of describing the maxima of a stochastic process in the limit where there are a large number of random variables.
For example, consider the process consisting of IID centered Gaussian random variables. 
One realization of this process for $2^{10}$ variables is depicted in Figure \ref{fig: iid} with a particular choice of the variance.
The maximum of the rugged landscape of this realization lies around $6$. 
If we repeat this exercise for several realizations, we would observe that this is the case for every realization:
the maximum has fluctuations roughly of order one around a deterministic value which is close to 6.

\begin{figure}[h]
\includegraphics[height=6cm]{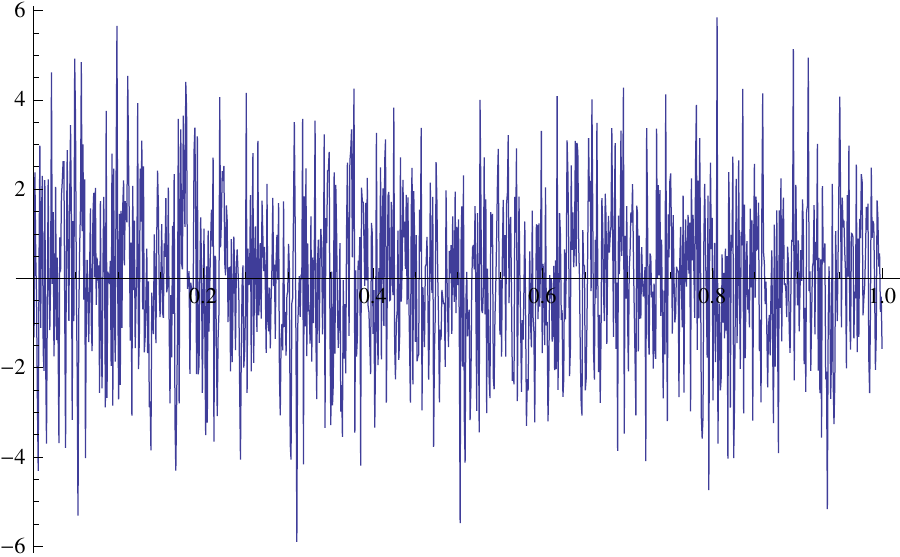}
\caption{A realization of $2^{10}$ centered Gaussian random variables of variance $\frac{1}{2}\log 2^{10}$.}
\label{fig: iid}
\end{figure}

Now consider a drastically different stochastic process: let $\rm U$ be a random $2^{10}\times2^{10}$ unitary matrix sampled uniformly from the unitary group
and consider the modulus of the characteristic polynomial on the unit circle 
$$
{\rm P}_{\rm U}(\theta)=\left|\det(e^{i\theta}-{\rm U})\right|\ .
$$
A realization of the logarithm of this quantity on the interval $[0,2\pi]$ is given in Figure \ref{fig: unitary}.
\begin{figure}[h]
\includegraphics[height=6cm]{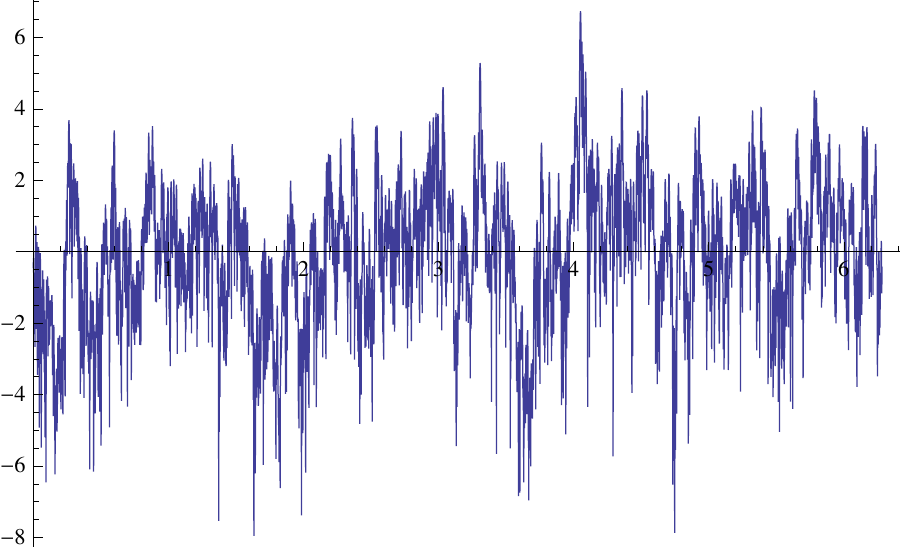}
\caption{A realization of ${\rm P}_{\rm U}(\theta)$ on $[0,2\pi]$ of a random unitary matrix $U$ of size $2^{10}\times 2^{10}$.}
\label{fig: unitary}
\end{figure}
This random process is of course qualitatively very different  from the IID process above. 
It is not Gaussian and it is a continuous process with singularities at the eigenvalues of $U$. 
However, perhaps by coincidence,  the maximum of the landscape lies also around $6$.
In fact, if we repeat the experiment, we would again observe that the maximum fluctuates around a deterministic value close to $6$.
If we were precise enough, we would see that the deterministic value is close to the one for IID variables but there is a slight discrepancy. 
The conjecture of Fyodorov-Hiary-Keating \cite{fyodorov-hiary-keating, fyodorov-keating} makes a detailed prediction 
for the maximum of ${\rm P}_{\rm U}(\theta)$. The distribution of the maximum is different in the limit of a large number of variables from the one of IID,
though it is close. However, it is a perfect match with the distribution of log-correlated Gaussian field. 
This conjecture is still open (even to first order). 
However, the mechanism that would be responsible for this conjecture to hold is now quite clear.
This is the goal of this review to illustrate the techniques for log-correlated fields in the forthcoming sections and their relations to ${\rm P}_{\rm U}(\theta)$ in 
Section \ref{sect: universality}. 

Throughout the paper, the framework will be as follows. 
The random field will be indexed by a discrete space $V_n$ with $2^n$ points. 
One important example is $V_n=\{-1,1\}^n$ representing the configurations of $n$ spins. 
In most examples, $V_n$ will be equipped with a metric that we denote $d(v,v')$, $v,v'\in V_n$.
We consider the random field
$$
X(n)=\{X_v(n), v\in V_n\} \ .
$$
defined on a probability space $(\Omega, P)$. 
We will suppose without loss of generality that $E[X_v(n)]$ for all $v\in V_n$. 
In most examples, the variance will be the same for every $v\in V_n$.
This is not true for the two-dimensional Gaussian free field because of the boundary effect.
However, the variance is almost uniform for most points which ultimately does not affect the results, but complicate the analysis.
With this in mind, we choose to parametrize the variance as
$$
E[X_v(n)]=\sigma^2 n\ ,
$$
for some $\sigma^2$. 
This turns out to be the correct physical parametrization: it gives a a non-trivial limit for the free energy per spin in the limit $n\to\infty$, cf. Section \ref{sect: stat phys}.

Even though the analysis we present are not limited to Gaussian fields, it is convenient to present the ideas in this setting.
There are many advantages to do so. First, the distribution of the field $X(n)$ is then entirely determined by the covariances 
$$
E[X_v(n)X_{v'}(n)]\ .
$$
Second, estimates for the probability of large values for Gaussians are very precise: it is not hard prove the Gaussian estimate 
for $X_v(n)\sim \mathcal N(0,\sigma^2n)$,
\begin{equation}
\label{eqn: gaussian estimate}
P(X_v(n)>a)=\frac{(1+o(1))}{\sqrt{2\pi}}\ \frac{\sqrt{\sigma^2n}}{ a}\ \exp\left({-\frac{a^2}{2\sigma^2n}}\right) \qquad \text{for $a>\sqrt{\sigma^2n}$.}
\end{equation}
Finally, comparison results for Gaussian processes can be used in some instances to bound functionals of the processes above and below with functionals of processes with simpler correlation structures. 

The central question of extreme value theory is to describe the functional
$$
\max_{v\in V_n} X_v(n) \text{ in the limit $n\to\infty$.}
$$
As for other global functionals of a stochastic process such as the sum, it is reasonable to expect universality results for the limit, i.e.~ results that depend mildly on the fine details of the distributions of $X(n)$. 
From a physics perspective, this question corresponds to describing the ground state energy of a system with state space $V_n$ and energy
$-X_v(n)$ for the state $v$. The system is said to be {\it disordered} because the energies depend on the realization of the sample space $\Omega$. 
A realization of the process as depicted in Figures \ref{fig: iid} and \ref{fig: unitary} from this point of view represents the energy landscape of the system for a given disorder.

With a perhaps naive analogy with the limit theorems for sum of IID random variables, one might expect a corresponding {\it law of large numbers} and {\it central limit theorems} for the maximum. More precisely, the problem is to:
\begin{equation}
\label{eqn: question}
\boxed{
\text{\it Find a recentering $a_n$ and rescaling $b_n$ such that $\max_{v\in V_n} \frac{X_v(n)-a_n}{b_n}$ converges in law.}
}
\end{equation}
If these sequence exist, it suggests that there is a non-trivial random variable $\mathcal M$ independent of $n$ such that the approximate 
equality holds:
$$
\max_{v\in V_n} X_v(n)\approx a_n + b_n \mathcal M\ .
$$
In this regard, the recentering $a_n$ is the term analogous to the law of large numbers whereas $b_n$
sizes the magnitude of the fluctuations with distribution $\mathcal M$, thus belongs to the central limit theorem regime.
Of course, $a_n$ should be unique up to an additive constant and $b_n$, up to a multiplicative constant. 
We refer to $a_n$ as the {\it order of the maximum}.
We shall focus on general techniques to rigorously obtain the order of the maximum $a_n$ for log-correlated fields in Section 3.
The proofs of fluctuations are technically more involved, but the heuristics developed for the order of the maximum are crucial 
to the analysis at a more refined level.

The starting point of our analysis is the result when $X_v(n)$ are IID Gaussians with variance $\sigma^2n$.
In this case, the problem \eqref{eqn: question} is a simple exercise and has a complete answer. 
\begin{prop}
\label{prop: iid}
Let $X(n)=(X_v(n), v=1,\dots, 2^n)$ be IID centered Gaussians with variance $\sigma^2n$.
Then for $c=\sqrt{2\log 2} \sigma$ and 
\begin{equation}
\label{eqn: iid a_n b_n}
a_n=c n -\frac{1}{2}\frac{\sigma^2}{c}\log n, \qquad b_n=1,
\end{equation}
we have
$$
\lim_{n\to\infty}P\left(\max_{v\in V_n} X_v(n)\leq a_n+b_n x\right)= \exp(-Ce^{-cx})\ \ , \text{ with $C=\frac{\sigma/c}{\sqrt{2\pi}}$.}
$$
\end{prop}
In other words, the maximum converges in distribution to a {\it Gumbel random variable} also known as {\it double-exponential}.
The order of the maximum is linear in $n$ at first order with velocity $c=\sqrt{2\log 2} \sigma$. 
The subleading order is logarithmic in $n$. 
The proof of this is simple but is instructive for the case of log-correlated fields.
The same type of arguments can be used for the {\it order statistics}, i.e.~ the joint distribution of first, second, third, etc maxima.
It converges to a Poisson point process with exponential density.
\begin{proof}[Proof of Proposition \ref{prop: iid}]
Clearly, since the random variables are IID, we have
$$
P\left(\max_{v\in V_n} X_v(n)\leq a_n+b_n x\right)
=\left(1-\frac{2^n P\left( X_v(n)> a_n+b_n x\right)}{2^n}\right)^{2^n}\ .
$$
This implies that the convergence of the distribution reduces to show that
$$
2^nP\left( X_v(n)> a_n+b_n x\right) \text{ converges for an appropriate choice of $a_n$ and $b_n$.}
$$
The right side is exactly the expected {\it number of exceedances}, that is 
$$
E\left[\mathcal N(y)\right], \qquad \text{ where }\mathcal N(y)=\#\{v\in V_n: X_v(n)>y\}
$$
for $y=a_n+b_nx$. It is not hard to prove using the Gaussian estimate \eqref{eqn: gaussian estimate} that the choice $a_n=c n -\frac{1}{2}\frac{\sigma^2}{c}\log n$ and $b_n$ ensures the convergence. 
In fact, the leading order $c n$ lies in the large deviation regime for the distribution $X_v(n)$. It compensates exactly the {\it entropy} term $2^n$. As for the subleading order $-\frac{\sigma^2}{2c}$, it balances out the fine asymptotics
$$
\frac{\sigma^2 n}{a_n}=\frac{\sqrt{\sigma^2 n}}{cn(1+o(1))}=O(n^{-1/2})
$$
in the Gaussian estimate. 
\end{proof}

It is a real challenge to achieve precise results of the kind of Proposition \ref{prop: iid} for a stochastic $X(n)$ with correlations of the order of the variance.
Correlations should affect the order of the maximum and its fluctuations, but to which extent ?
It is reasonable to expect the statistics of extremes of such processes to retain features of the IID statistics if we look at a coarse-enough scale. 
When correlations are present the expected number of exceedances is no longer a good proxy for the correct level of the maximum, since by linearity of the expectations, it is blind to correlations and is the same as IID.
The expected number of exceedances turns out to be inflated by the rare events of exceedances far beyond the right level of the maximum.
As we will see in Section \ref{sect: order}, the number of exceedances need to be modified to reflect the correct behavior of the maxima.

\subsection{Log-correlated fields}

In the framework where $X(n)=(X_v(n), v\in V_n)$ is a stochastic process on a metric space $V_n$ with distance $d$, 
a {\it log-correlated} field has correlations
\begin{equation}
\label{eqn: log}
E[X_v(n)X_{v'}(n)]\approx -\log d(v,v')\ .
\end{equation}
In particular, the correlation decays very slowly with the distance! Slower than any power. 
We emphasize that the covariance might not be exactly logarithmic for the field to exhibit a log-correlated behavior. 
For one, the covariance matrix must be positive definite, hence the covariance might be exactly the logarithm of the distance.
Moreover, the logarithmic behavior might only be exhibited in the bulk due to some boundary effect. 
As we will see in Section \ref{sect: examples}, this is the case for the two-dimensional Gaussian free field. 
Other important examples of processes in this class are branching random walks and models of Gaussian multiplicative chaos.
It is good to re-express \eqref{eqn: log} in terms of the size of neighborhoods with a strong correlations with a fixed point $v$.
More precisely, for $r>0$, consider the number of points $v'$ whose covariance is a fraction $r$ of the variance.
The logarithmic nature implies that
\begin{equation}
\label{eqn: nbhd}
\frac{1}{2^n}\#\big\{v'\in V_n: \frac{E[X_v(n)X_{v'}(n)]}{E[X_v(n)^2]} > r \big\}\approx 2^{-rn}\ .
\end{equation}
In other words, it takes approximately $2^{rn}$ balls to cover the space $V_n$ with neighborhoods where correlations between points is greater than $r$ times the variance. 
In particular, the size of these neighborhoods lies at the mesoscopic scales compared to the size of the systems.
This will play an important role in the analysis of Sections \ref{sect: examples} and \ref{sect: order}. 

The motivations for the study of log-correlated fields are now plenty. 
On one hand, in the physics literature, the work of Carpentier \& Ledoussal \cite{carpentier-ledoussal} spurred a lot of interest in the study of such processes as energy landscapes of disordered systems. We mention in particular the works of Fyodorov \& Bouchaud \cite{fyodorov-bouchaud} and Fyodorov {\it et al.}~\cite{fyodorov-ledoussal-rosso}. These papers develop a statistical mechanics approach to the problem of describing the extreme value statistics of the systems that are applicable to a wide range of systems. 
Second, these processes play an essential role in Liouville quantum gravity as well as models of three-dimensional turbulence. We refer to the excellent review on Gaussian Multiplicative Chaos of Rhodes \& Vargas \cite{rhodes-vargas_review} and references therein for details in these directions.
Third, many models of volatility in finance are now built on the assumption that the time-correlations of the returns are log-correlated, see also \cite{rhodes-vargas_review}  for more details on this.
Finally, as we shall see in Section \ref{sect: universality}, log-correlated fields seem to provide the right structure to study the local maxima of the Riemann zeta function on the critical line as suggested in \cite{fyodorov-hiary-keating, fyodorov-keating}.  

It is argued in the physics literature that the distribution of the maximum of log-correlated fields is a borderline case where the features of IID statistics should still be apparent. In fact, it is expected that the following should hold for log-correlated fields that are close to Gaussians in some suitable sense. 
\begin{conj}[Informal]
\label{conj: informal}
Let $X(n)=\{X_v(n), v\in V_n\}$ be a log-correlated field in the sense of \eqref{eqn: log} with $E[X_v(n)]=0$ and $E[X_v(n)^2]\approx \sigma^2 n$.
Then for $c=\sqrt{2\log 2} \sigma$ and 
\begin{equation}
\label{eqn: log a_n b_n}
a_n=c n -\frac{3}{2}\frac{\sigma^2}{c}\log n, \qquad b_n=1,
\end{equation}
we have
$$
 \lim_{n\to\infty}P\left(\max_{v\in V_n} X_v(n)\leq a_n+b_n x\right)= E[e^{-CZ e^{-cx}}]\ .
$$
for some constant $C$ and random variable $Z$ called the {\it derivative martingale}. 
\end{conj}
This is to be compared with Proposition \ref{prop: iid} for IID random variables.
Perhaps surprisingly, the first order of the maximum is the same as for IID with the same velocity. 
The correlations start having an effect in the subleading where $\frac{1}{2}$ for IID is changed to $\frac{3}{2}$.
This seemingly small change hides an important mechanism of the extreme value statistics that will be explained in Section \ref{sect: subleading}.
For now, let us just observe that under the recentering $a_n=c n -\frac{3}{2}\frac{\sigma^2}{c}\log n$, the expected number of exceedances diverges like $n$.
 To get the correct order, it will be necessary to modify the exceedances with a description of the values at each scale of the field. 
Finally, the fluctuations are not exactly Gumbel as for IID Gaussians. 
The distribution is {\it a mixture of Gaussians}, the mixture being on the different realizations of the random variables. 
This in fact changes qualitatively the nature of the distribution as it is expected that the right tail behaves like $xe^{-cx}$ for $x$ large whereas as it is $e^{-cx}$ for double exponential distribution.

Conjecture \ref{conj: informal} was proved in many instances of log-correlated fields.
The most important contribution is without a doubt the seminal work of Bramson \cite{bramson} who proved the result for branching Brownian motion (see Lalley \& Sellke \cite{lalley-sellke} for the expression in terms of the derivative martingale).
For branching random walks, it was proved in great generality by A\"idekon \cite{aidekon}, see also \cite{addario-reed, bachmann, bramson-ding-zeitouni2};
for the two-dimensional Gaussian free field, by Bramson {\it et al.}~\cite{bramson-ding-zeitouni} (see Biskup \& Louidor \cite{biskup-louidor} for the expression in terms of the derivative martingale), for a type of Gaussian Multiplicative Chaos by Madaule \cite{madaule}, and finally in great generality for log-correlated Gaussian fields by Ding {\it et al.}~\cite{ding-roy-zeitouni}.

We will not have time in this review paper to touch the subject of the order statistics or {\it extremal process} of log-correlated fields. 
Again, it turns out that the statistics retain features of the IID case. In fact, it is expected that the extremal process is a Poisson cluster process or Poisson decorated point process. This was proved for branching Brownian motion in \cite{arguin-bovier-kistler_genealogy,arguin-bovier-kistler_poisson,arguin-bovier-kistler_extremal,abbs}, for branching random walks in \cite{madaule_brw}, and partially for the two-dimensional Gaussian free field in \cite{biskup-louidor}.

\subsection{Relations to statistical physics}
\label{sect: stat phys}
Questions of extreme value statistics such as Conjecture \ref{conj: informal} can be addressed from a statistical physics point of view. 
In that case the object of interest is the {\it partition function}
\begin{equation}
\label{eqn: partition}
Z_n(\beta)=\sum_{v\in V_n} \exp(\beta X_v(n)) , \ \beta>0.
\end{equation}
(Observe that we use $\beta$ as opposed to the customary $-\beta$ since we are interested in the maximum and not the minimum of $X(n)$.)
By design, the points $v$ with a high value of the field have a greater contribution to the partition function, and the continuous parameter $\beta$ (the inverse temperature) adjusts the magnitude of the contribution. 
The {\it free energy per particle} or {\it normalized log-partition function} is also of importance:
\begin{equation}
\label{eqn: free energy}
f_n(\beta)=\frac{1}{\beta n}\log Z_n(\beta)\ .
\end{equation}
In particular, in the limit $n\to\infty$, the free energy contains the information on the first order of the maximum.
Indeed, we have the elementary inequalities
$$
\frac{\max_{v\in V_n} X_v(n)}{n}\leq f_n(\beta)\leq \frac{\log 2}{\beta} + \frac{\max_{v\in V_n} X_v(n)}{n}\ ,
$$
therefore
$$
\lim_{n\to\infty} \frac{\max_{v\in V_n} X_v(n)}{n}=\lim_{\beta\to\infty}\lim_{n\to\infty} f_n(\beta), \ \text{ whenever the limits exist.}
$$
To obtain finer information on the maximum such as subleading orders or fluctuations, 
one could study directly the partition functions $Z_n(\beta)$. 
For example, one could compute the moments of $Z_n(\beta)$ and try to infer the distribution of $Z_n(\beta)$ from it. 
This is the approach taken in  \cite{fyodorov-bouchaud, fyodorov-ledoussal-rosso} for example, 
but there are major obstacles to overcome to make this type of arguments fully rigorous. 

For log-correlated fields, it is possible to compute rigorously the free energy by the Laplace method.
We have by rearranging the sum over $v$ that
\begin{equation}
\label{sect: laplace}
f_n(\beta)\approx \frac{1}{n} \log\left(\sum_{E\in [0,c]} \exp(\beta n E + nS_n(E))\right)\ ,
\end{equation}
where 
\begin{equation}
\label{sect: log number}
S_n(E)=\frac{1}{n}\log \mathcal N_n(En)=\frac{1}{n}\log\#\{v\in V_n: X_v(n)> En\}\ .
\end{equation}
is the {\it entropy} or {\it log-number of high points}, i.e.~ the log-number of exceedances at level $En$. 
In the limit, it is possible to prove in some cases, see e.g. \cite{daviaud} and \cite{arguin-zindy, arguin-zindy2}, that for log-correlated fields
\begin{equation}
\label{entropy}
S(E)=\lim_{n\to \infty} S_n(E)=
\begin{cases}
0 \ &\text{ if $E\geq c$}\\
\log 2 - \frac{E^2}{2\sigma^2} \ &\text{ if $E\in[0,c]$}
\end{cases}
\text{ in probability,}
\end{equation}
where $c=\sqrt{2\log 2} \sigma$. We will see a general method to prove this in Section \ref{sect: first}.
Note that this is exactly the result one would obtain for IID random variables by applying \eqref{eqn: gaussian estimate}!
The free energy is then easily calculated like a Gibbs variational principle
\begin{equation}
\label{eqn: log free energy}
\lim_{n\to\infty} f_n(\beta)=
\max_{E\in [0,c]}\left\{\beta E +\log 2 - \frac{E^2}{2\sigma^2}\right\}
=
\begin{cases}
\log 2 +\frac{\beta^2 \sigma^2}{2} \ &\text{ if $\beta < \frac{c}{\sigma^2}$}\\
c\beta \ &\text{ if $\beta \geq \frac{c}{\sigma^2}$.}
\end{cases}
\end{equation}
Again, not suprisingly, the free energy of log-correlated fields corresponds to the one for IID variables also known as the {\it Random Energy Model} (REM) as introduced by Derrida \cite{derrida}. In particular, the free energy exhibits a {\it freezing} phase transition, i.e.~ for $\beta>\frac{c}{\sigma^2}$, 
the free energy divided by $\beta$ is constant reflecting the fact the partition function is supported on the maximal values of the field.
Equation \eqref{eqn: log free energy} was first proved for branching random walks by Derrida \& Spohn \cite{derrida-spohn} 
who pioneered the study of log-correlated random fields as disordered systems. 
Equation \eqref{eqn: log free energy} was also proved  for other log-correlated Gaussian fields in \cite{arguin-zindy, arguin-zindy2}. 

Another important object in statistical physics is the {\it Gibbs measure} which in our context is a random probability measure on $V_n$ defined as
\begin{equation}
\label{eqn: gibbs}
G_{n,\beta}(v)=\frac{\exp \beta X_v(n)}{Z_n(\beta)}\ , v\in V_n.
\end{equation}
This contains refined information on the order statistics of the field $X(n)$. 
We only mention here that the limiting distribution of the Gibbs measure depends on the parameter $\beta$. 
In particular, the {\it annealed Gibbs measure}
$$
G^{\text ann}_{n,\beta}(v)=\frac{\exp \beta X_v(n)}{E[Z_n(\beta)]}
$$
should have a non-trivial continuous limit for high temperature $\beta<\frac{c}{\sigma^2}$ \cite{rhodes-vargas_review}, whereas for low temperature $\beta\geq \beta_c$, the quenched Gibbs measure as in \eqref{eqn: gibbs} is the relevant limiting object and should be singular \cite{arguin-zindy, arguin-zindy2}.


\section{Examples and general properties}
\label{sect: examples}

Branching random walk serves as a guiding example to prove rigorous results for log-correlated fields.
In particular, we will use it to illustrate three important properties of log-correlated fields. 
We will refer to these as the {\it multi-scale decomposition}, {\it the dichotomy of scales} and {\it the self-similarity of scales}. 
As we will see these properties do not hold exactly for general log-correlated fields, like the 2D Gaussian free field.
However, they hold approximately enough to reproduce the extremal behavior.

\subsection{Branching random walk}
\label{sect: brw}

We will focus on a simple example of branching random walk (BRW), where the branching is binary and the increments are Gaussian.
This process is sometimes called {\it hierarchical Gaussian field}.
More specifically, let $V_n$ be the leafs of a binary rooted tree with $n$ generations. 
The field is indexed by $V_n$ and thus has $2^n$ random variables. As before, we denote the field by
$$
X(n)=\{X_v(n): v\in V_n\}\ .
$$
The random variables $X_v(n)$ are constructed as follows. 
We consider $\{Y_e\}$ IID centered Gaussian random variables of variance $\sigma^2$ indexed by the edges of the binary tree.
The parameter $\sigma^2$ will be adjusted when we compare log-correlated fields to approximate branching random walks.
For a given leaf $v\in V_n$, the random variable $X_v(n)$ is given by the sum of the $Y_e$'s for the edges $e$ on the path from the root $\emptyset$ to the leaf $v$, i.e.~, 
$$
X_v(n)=\sum_{e:\ \emptyset \to v} Y_e\ .
$$
We will often abuse notation and write $Y_v(l)$ for the variable $Y_e$ for the leaf $v$ at level $l$ in the tree. In other words,
\begin{equation}
\label{eqn: X brw}
X_v(n)=\sum_{l=1}^n Y_v(l)\ .
\end{equation}
We refer to $Y_l(v)$ as the {\it increment at scale $l$} of $X_v(n)$. 
With this definition, it is easy to see that the variance of each variable is
$$
E[X_v(n)^2]=\sum_{l=1}^n E[Y_v(l)^2]=\sigma^2 n\ .
$$
As for the covariance for $v,v'\in V_n$, we need to define
$$
v\wedge v' \text{ \it, the level of the tree where the paths of $v$ and $v'$ to the root split.}
$$
We call $v\wedge v'$ the {\it branching time} or {\it branching scale} of $v$ and $v'$. Note that $0\leq v\wedge v'\leq n$. 
See Figure \ref{fig: brw} for an illustration.
It then follows directly from the definition \eqref{eqn: X brw} that
$$
E[X_v(n)X_{v'}(n)]=\sum_{l=1 }^{v\wedge v'} E[Y^2_v(l)]=\sigma^2 \ (v\wedge v')\ .
$$
The covariance matrix completely determines the distribution of the centered Gaussian field $X(n)$. 

\begin{figure}[h]
\includegraphics[height=6cm]{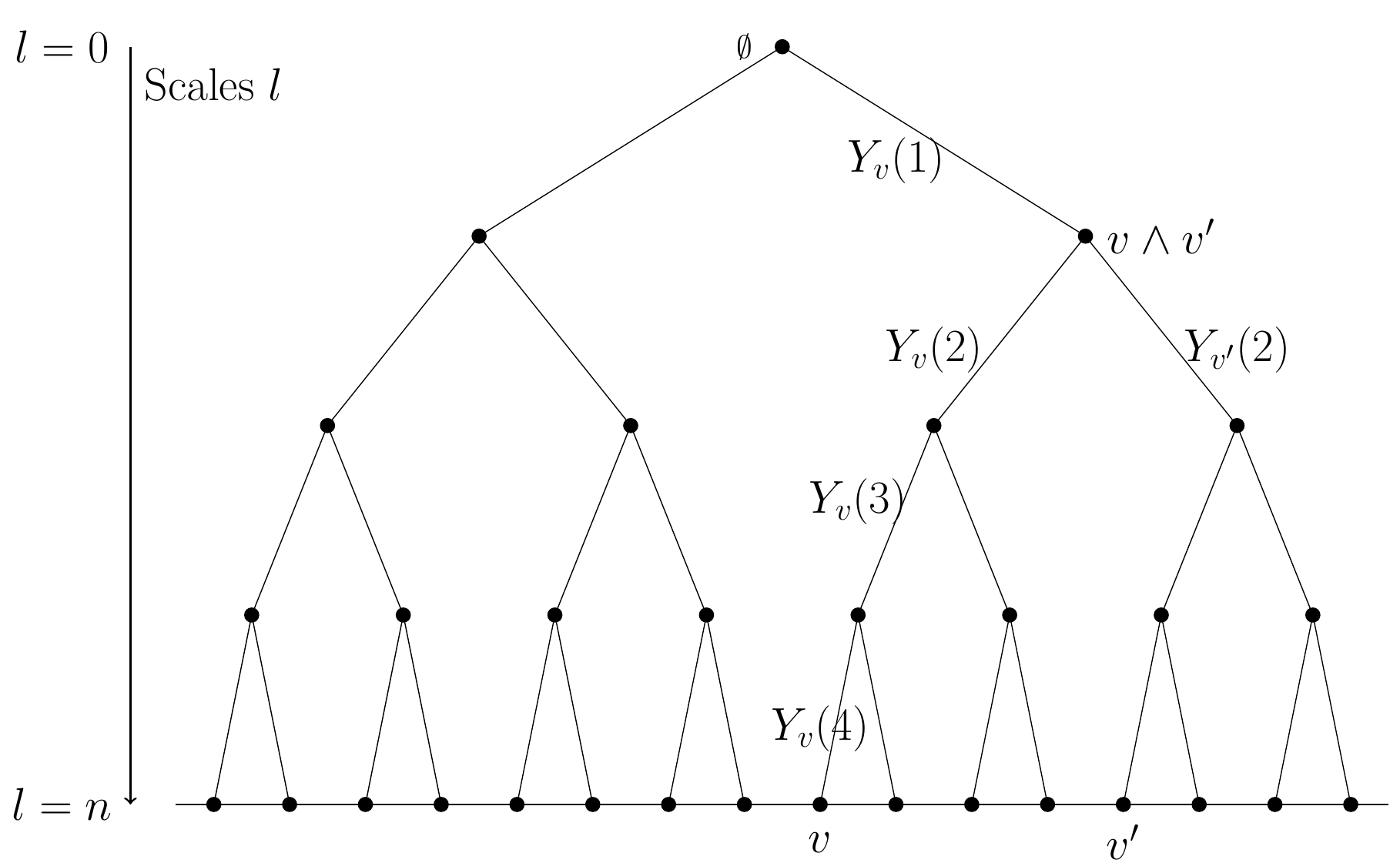}
\caption{An illustration of a binary branching random walk with increments $Y_v(l)$ at each scale $0\leq l\leq n$.}
\label{fig: brw}
\end{figure}

Why is the field $X(n)$ log-correlated ? The best way to see this is to think in terms of neighborhood with a certain covariance as in equation \eqref{eqn: nbhd}. Indeed, let $0\leq r\leq 1$, and suppose for simplicity that $rn$ is an integer. Then the size of the neighborhood of a given $v$ with covariance $r\sigma^2n$ with $v$ is exactly
$$
\#\big\{v'\in V_n: \frac{E[X_v(n)X_{v'}(n)]}{\sigma^2 n} \geq r \big\}=2^{n-rn}\ .
$$
It corresponds to those $v'$ which branched out from $v$ at level $rn$ !
It is a ball of size $2^{n(1-r)}$ and is of mesoscopic scale compared to the size of the systems.

There are three elementary properties of BRW which will have their counterpart for any log-correlated field. 
These properties will be key to guide our analysis in more complicated models. 
\begin{itemize}
\item {\it Multiscale decomposition:}
The first property simply refers to the decomposition \eqref{eqn: X brw} into independent increments. 
$$
X_v(n)=\sum_{l=1}^n Y_v(l)\ .
$$
\item {\it Dichotomy of scales:}
Another simple consequence of the definition of BRW is the fact that for scales before the branching points of $v$ and $v'$, the increments are perfectly correlated whereas after the branching point, they are exactly independent. In other words:
\begin{equation}
\label{eqn: dichotomy}
E[Y_l(v)Y_l(v')]
=\begin{cases}
\sigma^2 &\text{if $ l\leq v\wedge v' $,}\\
0 &\text{if $ l>v\wedge v' $\ .}
\end{cases}
\end{equation}
We refer to this abrupt change in correlations for the increments as the {\it dichotomy of scales}. 
\item {\it Self-similarity of scales:}
In the case of BRW, the increments at each scale are identically distributed.
In particular, together with the property of dichotomy, this property implies that for a given $0<l<n$, 
the variables $X_v(l)=\sum_{k=1}^l Y_v(k)$ define a process $X(l)$ with $2^l$ distinct values. 
Moreover, this process is exactly a BRW with $l$ levels. Similarly, for a given $v$, we can define the variables
$$
\{X_{v'}(n)-X_{v'}(l): v'\wedge v\geq l\}\ .
$$
This defines a BRW on $2^{n-l}$ points corresponding to the subtree with the common ancestor at level $l$ acting as the root.
\end{itemize}

We point out that the above construction leads to a natural generalization of fields by modifying the variance of the increments in the multiscale decomposition at each scale. The self-similarity is lost in these models. The study of such types of models go back to Derrida who introduced the {\it Generalized Random Energy Models} \cite{derrida_grem}.
In the case of BRW, these are called time-inhomogeneous BRW. They were studied in \cite{FangZeitouni2012,FangZeitouni2012JSP,MaillardZeitouni2015, BovierHartung14, BovierHartung2015}.
A similar {\it scale-inhomogenous} field for the Gaussian free field was introduced and studied in \cite{arguin-zindy2, arguin-ouimet}.

\subsection{2D Gaussian free field}
The discrete two-dimensional Gaussian free field (2DGFF) is an important model of random surfaces. 
Let $V_n$ be a finite square box of $\Z^2$. To emphasize the similarities with BRW, we suppose that $\#V_n=2^n$.
The 2DGFF is a centered Gaussian field $X(n)=\{X_v(n), v\in V_n\}$ defined as follows. 

Let
$\mathbb P_v$ be the law of a simple random walk $(S_k,k\geq 0)$ starting at $v\in V_n$ and $\mathbb E_v$, the corresponding expectation. 
Denote by $\tau_n$ the first exit time of the random walk from $V_n$. We consider the covariance matrix given by 
\begin{equation}
\label{eqn: gff cov}
E[X_v(n)X_{v'}(n)]=\mathbb E_v\left[\sum_{k=0}^{\tau_n} 1_{S_k=v'}\right]=G_n(v,v')\ .
\end{equation}
The right-hand side is the expected number of visits to $v'$ of a random walk starting at $v$ before exiting $V_n$. 
It is not hard to show, using the Markov propery of the random walk, that the right-hand side is the Green's function of the discrete Laplacian on $V_n$ with Dirichlet boundary conditions. 
More precisely, for $f$ a function on the vertices of $V_n$, the discrete Laplacian is
$$
-\Delta f(v)=\frac{1}{4}\sum_{\omega\sim v} f(\omega)-f(v)=\mathbb E_v[f(S_1)-f(v)]\ ,
$$ 
where $\omega\sim v$ means that $\omega$ and $v$ share an edge in $\Z^2$.
The Green's function can be seen as the inverse of the discrete Laplacian in the sense that it satisfies
$$
-\Delta G_n(v,v')=\delta_{v'}(v)=
\begin{cases}
1 & \text{ if $v=v'$,}\\
0 & \text{ if $v\neq v'$.}
\end{cases}
$$
In particular, the Green's function is symmetric and positive definite, since the Laplacian is. 
The density of the 2DGFF is easily obtained from the covariance \eqref{eqn: gff cov}. 
Indeed, by rearranging the sum, we have
$$
\sum_{v\sim v'}(x_v-x_{v'})^2=\frac{1}{4}\sum_v x_v(-\Delta) x_{v}\ .
$$
Again, since the Green's function is the inverse of $-\Delta$, we have that
\begin{equation}
\label{eqn: gff density}
P(X_v(n)\in dx_v, v\in V_n)=\frac{1}{Z}\exp\left(-\frac{1}{8} \sum_{v\sim v'} (x_v-x_{v'})^2\right) \ ,
\end{equation}
for the appropriate normalization $Z$. 

Why is this field log-correlated ? It is reasonable to expect this since the Green's function of the Laplacian in the continuous setting decays logarithmically with the distance. This analogy can be made rigorous using random walk estimates. 
In fact, it is possible to show, see e.g. \cite{Lawler2010}, that
\begin{equation}
\label{eqn: green}
G_n(v,v')=E_{v}\left[\sum_{k=0}^{\tau_n} 1_{S_k=v'}\right]=E_v[a(v',S_{\tau_n})]-a(v,v')
\end{equation}
where $a(v,v)=0$ and if $v\neq v'$
$$
a(v,v')=\frac{1}{\pi}\log d(v,v')^2 + O(1) +O(d(v,v')^{-2})\ ,
$$
where  $d(v,v')$ is the Euclidean distance in $\Z^2$. 
In particular, this implies that for points not too close to the boundary, the field is  log-correlated: $G_n(v,v')=\frac{1}{\pi}\log \frac{2^n}{d(v,v')^2}+O(d(v,v')^{-2})$. Moreover, we get the following upper bound on the variance:
$$
G_{n}(v,v)\leq \frac{1}{\pi}\log 2^n + O(1)\ .
$$
Again, a matching lower bound follows from \eqref{eqn: green}, but only for points far from the boundary.
As first step towards a connection to BRW, these estimates already suggest that the BRW parameter $\sigma^2$ should be taken to be 
$$
\sigma^2\longrightarrow \frac{\log 2}{\pi}\ .
$$

To make the connection with an approximate BRW more precise, we shall need three fundamental facts.
First, it is not hard to prove from the density \eqref{eqn: gff density} that the field satisfies the {\it Markov property},
that is for any finite box $B$ in $V_n$, 
\begin{equation}
\label{eqn: markov}
\text{\it $\{X_v(n), v\in B\}$ is independent of the field in $B^c$ given the field on the boundary $\partial B$.}
\end{equation}
Here the boundary $\partial B$ refers to the vertices in $B^c$ that share an edge with a vertex in $B$.
This property essentially follows from \eqref{eqn: gff density} since the correlations are expressed in terms of nearest neighbors. 
Second, it follows from the Markov property that the conditional expectation of the field inside $v$ given the boundary is simply a linear combination of the field on the boundary. More precisely, we have
\begin{equation}
\label{eqn: gff cond}
E[X_v(n)| \{X_{v'}(n), v'\in B^c\}]=\sum_{\huge u\in \partial B} p_{u}(v) X_u(n)
\end{equation}
where $p_u(v)$ is the probability of a simple random walk starting at $v$ to exit $B$ at $u$.
The specific form of the coefficient of the linear combination is proved using the strong Markov property of the random walk. 
A similar argument shows that the process is self-similar. 
Namely, if we write $X_v(B)=E[X_v(n)| \{X_{u}(n), u\in\partial B\}]$ for the harmonic average of $X_v(n)$ on the boundary of $B$, we have
\begin{equation}
\label{eqn: self}
\{X_v(n)-X_v(B), v\in B\Big\} \qquad \text{is a 2DGFF on $B$,}
\end{equation}
in the sense that it is a Gaussian free field with covariance \eqref{eqn: gff cov} restricted to $B$. 

The third property shows that if $v$ and $v'$ are in $B$, then $X_v(B)$ and $X_{v'}(B)$ should be close, if they are not too close to the boundary of $B$. More precisely, say $B$ is a square box or a disc containing $2^{l}$ points for some $l$.
Estimates of the Green's function as in \eqref{eqn: green} can be used to show that, see e.g. Lemma 12 in \cite{bolthausen-deuschel-giacomin}, 
\begin{equation}
\label{eqn: common}
E[(X_v(B)-X_{v'}(B))^2]\leq  O(d(v,v') 2^{-l/2})
\end{equation}
whenever the distance of $v$ and $v'$ to the center of $B$ is of the order of $2^{(l-1)/2}$ or less, to avoid boundary effects.

The connection between the 2D Gaussian free field and an approximate BRW is made through Equations \eqref{eqn: markov}, \eqref{eqn: self}, and \eqref{eqn: common}. For each $v\in V_n$ and $0\leq l\leq n$, we define $[v]_l$, a neighborhood of $v$ containing $2^{n-l}$ points. For simplicity, suppose that $[v]_l$ is a square box. If such a neighborhood is not entirely contained in $V_n$, we define $[v]_l$ to be the intersection of the neighborhood with $V_n$. Note that these neighborhoods shrink as $l$ increases. By convention, we take $[v]_0=V_n$ and $[v]_n=\{v\}$. The boundary is denoted $\partial [v]_l$ and the union of $\partial [v]_l$ with its boundary will be denoted by $\overline{[v]}_l$.
The connection of 2DGFF with BRW is illustrated in Figure \ref{fig: GFF} and is as follows. For simplicity, we assume that all vertices $v$ are not too close to the boundary of $V_n$. 

\noindent{\it Multiscale decomposition:}
The decomposition here is simply a {\it martingale decomposition}.  Namely, we define the {\it field of $v\in V_n$ at scale $l$} as
$$
X_v(l)=E[X_v(n)| \{X_{v'}(n): v'\in [v]_l^c\}]=E[X_v(n)| \{X_{u}(n): u\in \partial [v]_l\}]\ ,
$$
where the second equality is by the Markov property. 
Note that for each $v\in V_n$, $(X_v(l), 0\leq l\leq n)$ is a martingale by construction.
Define the {\it increments at scale $l$} as the martingale difference
$$
Y_v(l)=X_v(l) - X_v(l-1) \ , \ 1\leq l\leq n\ .
$$
These increments are orthogonal by construction. Since the field is Gaussian,  so are the conditional expectations. In particular, the increments are independent of each other for a given $v$.  Thus, the decomposition
\begin{equation}
\label{eqn: brw}
X_v(n)=\sum_{l=1}^n Y_v(l) \ , \ v\in V_n\ ,
\end{equation}
seems so far to be the right analogue of the multiscale decomposition of BRW. It remains to check the other two properties.

\noindent{\it Self-similarity of scales:}
First, note that the increments are Gaussian by definition. The fact that the increments have approximately the same variance follows from  \eqref{eqn: self}. 
Indeed, for $0\leq l\leq n$, the variance of $X_v(n)-X_v(l)$ is by \eqref{eqn: green}
$$
E[(X_v(n)-X_v(l))^2]=(n-l)\frac{\log 2}{\pi} +O(1)\ ,
$$
Because $X_l(v)$ is orthogonal to $X_v(n)-X_v(l)$, we deduce that
$$
E[X_v(l)^2]=l \frac{\log 2}{\pi} +O(1)\ .
$$
In particular, since $X_v(l+1)$ admits the orthogonal decomposition $X_v(l)+Y_v(l+1)$, this implies 
$$
E[Y_v(l)^2]= \frac{\log 2}{\pi}+o(1)
$$
where $o(1)$ is summable, so the contribution of the sum over $l$ of the error is $O(1)$.
This justifies the choice $\sigma^2=\log 2/\pi$ for the comparison with BRW. 
Of course, here, the increments are not exactly identically distributed, but since the error is summable, it will not affect the behavior of the extremes. 

\noindent {\it Dichotomy of scales:}
It remains to establish the correlations of the increments between two distinct $v, v'$. 
For this purpose, it is useful to define the {\it branching scale} $v\wedge v'$ as follows
\begin{equation}
\label{eqn: gff branching}
v\wedge v'=\min\{0\leq l\leq n: \overline{[v]}_l\cap \overline{[v']}_l= \emptyset\},
\end{equation}
that is the smallest scale for which the neighborhoods of $v$ and $v'$ (including their boundaries) are disjoint.
Note that because the neighborhoods have size $2^{n-l}$, the branching scale $v\wedge v'$ is related to the distance by 
$$
d(v,v')^2\approx 2^{n-v\wedge v'}\ .
$$ 
For small scales, a direct application of the Markov property \eqref{eqn: markov} yields
$$
E[Y_v(l)Y_{v'}(l)]= 0\ \text{ if $l>v\wedge v'$.}
$$
This is because if $l>v\wedge v'$, then the neighborhoods of $v$ and $v'$ do not intersect and are contained in each other's complement.
For scales larger than the branching scale, we expect the increments to be almost perfectly correlated. 
To see this, note that by definition, for $l=v\wedge v'-1$, then the neighborhoods intersect but for larger neighborhoods, that is $l<v\wedge v'-1$, $v$ and $v'$ must be contained in each other's neighborhood. In particular, an application of \eqref{eqn: common} gives
$$
E[(X_v(l)-X_l(v'))^2]=O(1)\ .
$$
In other words, the vectors $X_v(l)$ and $X_l(v')$ are close in $L^2$. Since these vectors are sums of orthogonal increments, these increments must be also close. 
In particular, we get 
$$
E[Y_v(l)Y_{v'}(l)]= \frac{\log 2}{\pi}+o(1)\ \text{ if $l<v\wedge v'-1$,}
$$
where $o(1)$ is summable in $l$. 

This completes the construction of an approximate BRW embedded in 2DGFF. We remark that the dichotomy of scales is not as clean cut as the one of BRW in \eqref{eqn: dichotomy}. In particular, nothing precise can be said for the correlation between exactly at the branching scale $l=v\wedge v'$. 
It turns out that the coupling and decoupling of the increments do not have to be perfect as long as they occur fast enough. 
The same will be true for non-Gaussian models in Section \ref{sect: universality}.

\begin{figure}[h]
\includegraphics[height=8cm]{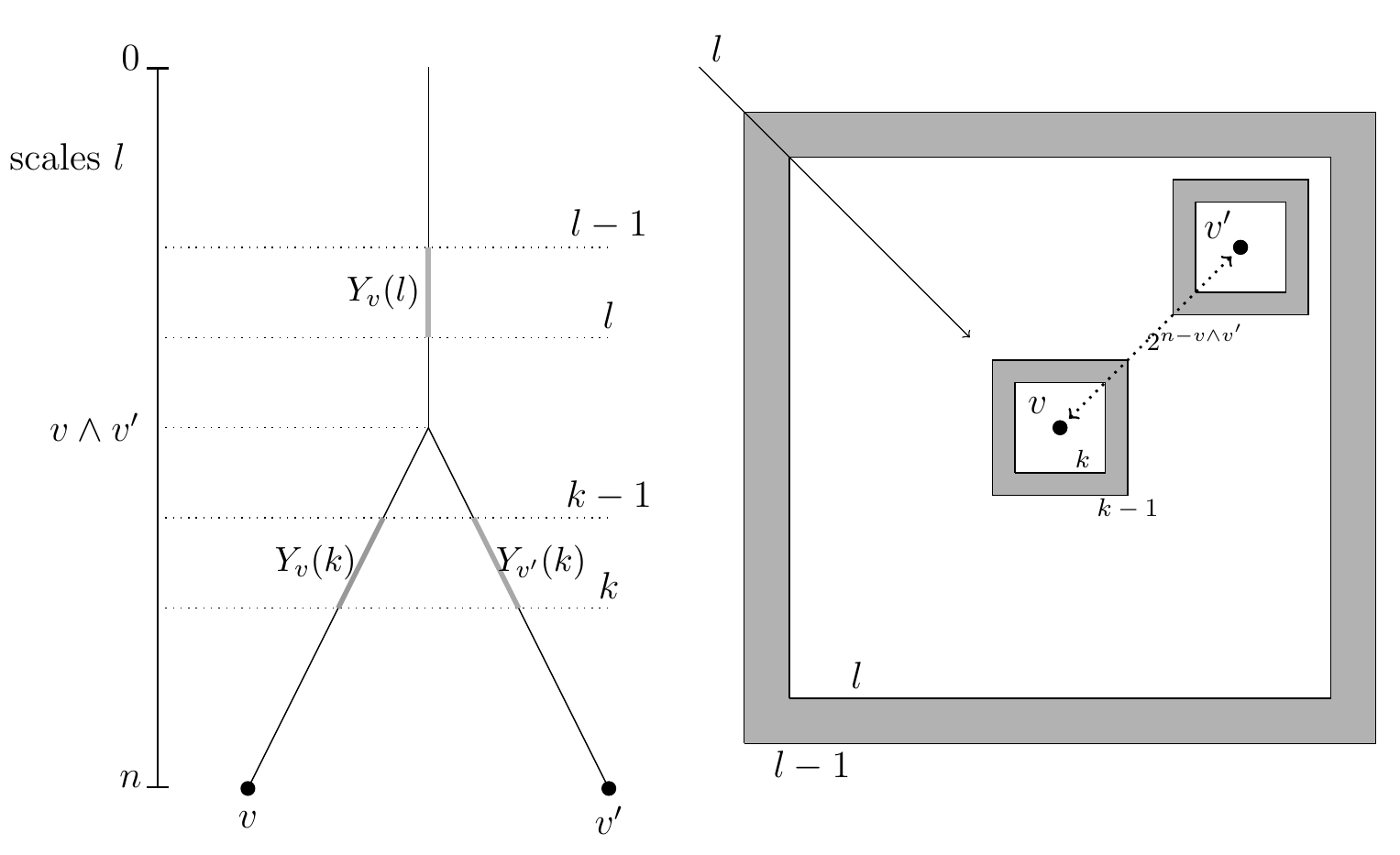}
\caption{An illustration of the approximate branching structure defined by the 2DGFF. The grey area represents the increments at each scale.}
\label{fig: GFF}
\end{figure}


\section{Order of the maximum}
\label{sect: order}

In this section, we explain a general method to show the leading and subleading orders  of the maximum for log-correlated fields in Conjecture \ref{conj: informal}. 
We will prove them in the case of Gaussian branching random walk on the binary tree, following the treatment in \cite{Kistler2014}.
We will also prove the convergence of the log-number of high points or entropy, cf.~\eqref{entropy}.
The heuristic is (or expected to be) the same for more complicated log-correlated models, such as the two-dimensional Gaussian free field. 
The multiscale decomposition of the field plays a fundamental role in the proof. 
This is also true when proving more refined results such as the convergence of the maximum and of the extremal process. 
In a nutshell, to prove such results on the maximum, one needs to understand the contribution of the increments at each scale to a large value of the field. 

\subsection{Leading order of the maximum}
\label{sect: first}
Recall the definition of the Gaussian BRW $X(n)$ on a binary tree defined in Section \ref{sect: brw}.
The following result yields the leading order of the maximum

\begin{thm}
\label{thm: first}
Let $X(n)=\{X_v(n): v\in V_n\}$ be a branching random walk on a binary tree with $n$ levels as in \eqref{eqn: brw} with Gaussian increments of variance $\sigma^2$.
Then
$$
\lim_{n\to\infty}\ \frac{\max_{v\in V_n} X_v(n)}{n}=\sqrt{2\log 2}\ \sigma \qquad  \text{in probability.}
$$
\end{thm}
The theorem goes back to Biggins for general BRW \cite{Biggins1976}. 
The analogue for the 2DGFF was proved by Bolthausen, Deuschel \& Giacomin \cite{bolthausen-deuschel-giacomin}. 
We point out that Kistler greatly simplified the proof in both cases \cite{Kistler2014} using a {\it multiscale refinement of the second moment method}.
This approach also makes the analysis amenable to non-Gaussian models. We refer the reader directly to \cite{Kistler2014} for a proof of the theorem.
We shall instead use Kistler's method to prove a very similar result on the {\it log-number of high points or exceedances}.

\begin{thm}
\label{thm: high}
Let $X(n)=\{X_v(n): v\in V_n\}$ be a branching random walk on a binary tree with $n$ levels as in \eqref{eqn: brw} with Gaussian increments of variance $\sigma^2$.
Then for $0\leq E< c=\sqrt{2\log 2} \sigma$,
$$
\begin{aligned}
\lim_{n\to\infty}\frac{1}{n}\log\#\{v\in V_n: X_v(n)> En\}
=\log 2\left(1 - \frac{E^2}{c^2}\right) \qquad
\text{ in probability.}
\end{aligned}
$$
\end{thm}
In other words, the number of points whose field value is larger than $E n$ is approximately $2^{n(1-E^2/c^2)}$ with large probability for $E$ smaller than the maximal level $c$. We stress that this is the same result for $2^n$ IID~Gaussian random variables of variance $\sigma^2$. 
The analogue of this result for 2DGFF was shown by Daviaud in \cite{daviaud}.

\begin{proof}

\noindent {\it Upper bound: } 
To bound the number of points beyond a given level, it suffices to use Markov's inequality. Let
$$
\mathcal N_n(En)=\#\{v\in V_n: X_v(n)> En\}\ .
$$
Then for a given $\varepsilon>0$, we have
\begin{equation}
\label{eqn: high markov}
P\left(\mathcal N_n(En)> 2^{n (1+\varepsilon)\left(1 - \frac{E^2}{c^2}\right)} \right)
\leq 2^{-n (1+\varepsilon)\left(1 - \frac{E^2}{c^2}\right)}E[\mathcal N_n(En)]\ .
\end{equation}
By linearity, the expectation is simply
$$
E[\mathcal N_n(En)]=2^nP(X_v(n)>En)\leq 2^n \exp\left({-\frac{E^2 n}{2\sigma^2}}\right)=2^{n(1-\frac{E^2 }{2c^2})}
$$
where we use the Gaussian estimate \eqref{eqn: gaussian estimate} with $X_v(n)\sim \mathcal N(0,\sigma^2n)$.
It follows that the probability in \eqref{eqn: high markov} goes to $0$. This proves the upper bound.

\noindent {\it Lower bound: }
It remains to show that
$$
P\left(\mathcal N_n(En)> 2^{n (1-\varepsilon)\left(1 - \frac{E^2}{c^2}\right)} \right)\to 1\ .
$$
The lower bound is an application of the {\it Paley-Zygmund inequality}, which states that for any random variable $\mathcal N>0$ and $0\leq \delta\leq 1$, 
\begin{equation}
\label{eqn: PZ}
P(\mathcal N>\delta E[\mathcal N])\geq (1-\delta)^2\frac{E[\mathcal N]^2}{E[\mathcal N^2]}
\end{equation}
So ultimately, one needs to prove that in the limit  the second moment matches the first moment square. 
This would certainly be the case if the variables were independent. In the case of BRW, this is not so. 
However, we can create enough independence and not losing too much precision by dropping the increments at lower scales.
In addition, as observed in \cite{Kistler2014}, the matching of the moments is greatly simplified if one considers increments rather than the sum of increments.
More precisely, since the leading of the maximum is linear in the scales, it is reasonable to expect that each scale contributes equally to the maximal value to leading order (since they have the same variance). It turns out this heuristic is correct and extend to the high points. 

With this in mind, we introduce a parameter $K$ which divides the $n$ scales coarser $K$ levels. 
We define accordingly, for $m=1,\dots, K$, the increments
$$
W_m(v)=\sum_{\frac{m-1}{K}< l \leq \frac{m}{K}} Y_l(v)\ .
$$
The idea is that for a high point at level $En$, each variable $W_m(v)$ should contribute $E\frac{n}{K}$. 
Without loss of generality, we can suppose $\varepsilon$ is small enough so that $(1+\varepsilon)E<c$.
We consider the 
event
$$
\mathcal E_m(v)=\{W_m(v)> (1+\varepsilon)\frac{n}{K}\ E \}\ , m=1,\dots,K,
$$
and the {\it modified number of exceedances}, 
$$
\widetilde{\mathcal N}_n=\sum_{v\in V_n} \prod_{m=2}^K1_{\mathcal E_m(v)}\ .
$$
Note that if $v$ is counted in $\widetilde{\mathcal N}_n$, then each $W_m(v)$ for $m\geq 2$ exceeds $(1+\varepsilon) K^{-1}nE$. 
Moreover, it is easy to check using the Gaussian estimate \eqref{eqn: gaussian estimate} that $|W_1(v)|$ is less than $K^{-1}n$ with large probability.
In particular, $X_v(n)>(1+\varepsilon) (1-K^{-1}) En$ on the events, so that we can take $K$ large enough depending on $\varepsilon$ so that
$X_v(n)>En$. We conclude that
$$
\widetilde{\mathcal N}_n\leq \mathcal N_n(En)\ ,
$$
hence it suffices to show 
$$
P\left(\widetilde{\mathcal N}_n> 2^{n (1-\varepsilon)\left(1 - \frac{E^2}{c^2}\right)} \right)\to 1\ .
$$
We apply the inequality \eqref{eqn: PZ} with $\delta=2^{-n\varepsilon\left(1 - \frac{E^2}{c^2}\right)}$. 
We have by the Gaussian estimate \eqref{eqn: gaussian estimate}
$$
 E[\widetilde{\mathcal N}_n]=2^n \frac{(1+o(1))}{\sqrt{2\pi}} \exp\left({-\frac{(1-K^{-1})(1+\varepsilon)^2E^2 n}{2\sigma^2}}\right)
\geq 2^{n \left(1 - \frac{E^2}{c^2}\right)}\ ,
$$
for $n$ large enough. 
Thus
$$
P\left(\widetilde{\mathcal N}_n> 2^{n (1-\varepsilon)\left(1 - \frac{E^2}{c^2}\right)} \right)\geq
P\left(\widetilde{\mathcal N}_n> \delta E[\widetilde{\mathcal N}_n]\right)
$$
It remains to show
\begin{equation}
\label{eqn: high to show}
\frac{(E[\widetilde{\mathcal N}_n])^2}{E[(\widetilde{\mathcal N}_n)^2]} \to 1 \ \text{as $n\to\infty$.}
\end{equation}
The idea is to split the second moment according to the branching scale. More precisely, we have
\begin{equation}
\label{eqn: sum split}
\begin{aligned}
E[(\widetilde{\mathcal N}_n)^2]
&= \sum_{\substack{v,v'\in V_n\\ v\wedge v'\leq \frac{n}{K}}} P(\bigcap_{m=2}^K \mathcal E_m(v)\cap \mathcal E_m(v'))+\sum_{r=2}^K \sum_{\substack{v,v'\in V_n\\ \frac{(r-1)n}{K}<v\wedge v'\leq \frac{rn}{K}}} P(\bigcap_{m=2}^K \mathcal E_m(v)\cap \mathcal E_m(v'))\\
&= \sum_{\substack{v,v'\in V_n\\ v\wedge v'\leq \frac{n}{K}}} \prod_{m=2}^K P( \mathcal E_m(v))^2+\sum_{r=2}^K \sum_{\substack{v,v'\in V_n\\ \frac{(r-1)n}{K}<v\wedge v'\leq \frac{rn}{K}}} \prod_{m=2}^KP( \mathcal E_m(v)\cap \mathcal E_m(v'))
\end{aligned}
\end{equation}
where the second line holds by the independence of the increments between the scale and the fact that $W_m(v)=W_m(v')$ for all $m\geq 2$ if $v\wedge v'\leq n/K$ by the dichotomy of scales. 
Since there are at least $2^n(2^n-2^{n-n/K})=2^{2n}(1-o(1))$ pairs such that $v\wedge v'\leq n/K$, the first term of \eqref{eqn: sum split}  is 
$(1-o(1))(E[\widetilde{\mathcal N}_n])^2$. The proof will be concluded once we prove that the second term is $o(1) (E[\widetilde{\mathcal N}_n])^2$.

Guided by the dichotomy of scales, we do not lose by dropping the constraints on $m$ smaller than the branching scales for one of the $v$'s since they are identical. The second term is smaller than
$$
\sum_{r=2}^K \sum_{\substack{v,v'\in V_n\\ \frac{(r-1)n}{K}<v\wedge v'\leq \frac{rn}{K}}} \prod_{m=2}^KP( \mathcal E_m(v))
\prod_{m\geq r+1}P( \mathcal E_m(v')) ,
$$
where we use the independence of the increments after the branching scales. 
Since there are at most $2^n(2^{n-\frac{(r-1)n}{K}})$ pairs $v\wedge v'$ such that $\frac{(r-1)n}{K}<v\wedge v'\leq \frac{rn}{K}$,
we have that the above is smaller than
$$
(E[\widetilde{\mathcal N}_n])^2\sum_{r=2}^K \left(2^{\frac{(r-1)n}{K}} \prod_{m=2 }^rP( \mathcal E_m(v))\right)^{-1}
\leq (E[\widetilde{\mathcal N}_n])^2\sum_{r=2}^\infty \left(2^{(r-1)nK^{-1}\big(1-(1+\varepsilon)^2E^2/c^2\big)}\right)^{-1}
$$
where we used the Gaussian estimate with variance $n/K$ for each variable $W_m$. 
The sum is geometric. It converges because of the choice $(1+\varepsilon)E<c$. Thus the second term of \eqref{eqn: sum split} is smaller than
$$
2(E[\widetilde{\mathcal N}_n])^2\ 2^{-nK^{-1}\big(1-(1+\varepsilon)^2E^2/c^2\big)}=o(1)(E[\widetilde{\mathcal N}_n])^2\ ,
$$
for $n$ large enough. This proves \eqref{eqn: high to show} and concludes the proof of the theorem.
\end{proof}

\subsection{Subleading order of the maximum}
\label{sect: subleading}

We now turn to the proof of the subleading order of the maximum for a Gaussian branching random walk.
The basic idea goes back to the seminal work of Bramson who showed the convergence of the law of the maximum of branching Brownian motion.
It consists in observing that the increments of the maximum must satisfy a linear constraint at each scale. 
This is consistent with the fact the first order is linear in the scales and that at each scale $l$, the variables $X_l(v)$ form themselves a BRW with $2^l$ points.
To prove convergence, more precise estimates are needed than for the subleading order. 
We follow here the treatment of  \cite{Kistler2014} with the additional use of the {\it ballot theorem} as in \cite{aidekon-shi}, cf. Theorem \ref{thm: ballot} below.
\begin{thm}
\label{thm: second}
Let $X(n)=\{X_v(n): v\in V_n\}$ be a branching random walk on a binary tree with $n$ levels as in \eqref{eqn: brw} with Gaussian increments of variance $\sigma^2$.
Then
$$
\lim_{n\to\infty}\ \frac{\max_{v\in V_n} X_v(n)-cn}{\log n}= -\frac{3}{2}\qquad  \text{in probability.}
$$
where $c=\sqrt{2\log 2}\ \sigma$.
\end{thm}
For simplicity, we define the deterministic displacement of the maximum as
\begin{equation}
m_n= cn -\frac{3\sigma^2}{2c}\log n, \text{ and more generally, }m_n(\varepsilon)=m_n+\varepsilon\log n\ .
\end{equation}
The first observation is to realize that the expected of exceedances at the level $m_n$ diverges!
Indeed, using the Gaussian estimate \eqref{eqn: gaussian estimate} and linearity
$$
E[\mathcal N_n(m_n)]=E[\#\{v\in V_n: X_v(n)>m_n\}]= O(n)\ .
$$
This is the first sign that the branching structure matters for the subleading order.
The divergence of the expectation comes from atypical events that inflate the expectation.
Therefore, it is necessary to restrict the exceedances to the typical behavior values. 
Since the first order is linear in the scales and the variables $\{X_l(v): v\in V_n\}$ form a BRW on $2^l$ values,
it is reasonable to expect that $X_l(v)=\sum_{k=1}^lY_k(v)\leq cl +B$ for some appropriate choice of {\it barrier} $B$.
Keeping this in mind, we introduce a {\it modified number of exceedances} 
$$
\widetilde{\mathcal N}_n=\#\{v\in V_n: X_v(n)>m_n, X_l(v)\leq cl + B\ \forall l\leq n\}\ .
$$
It turns out that $E[\widetilde{\mathcal N}_n]=O(1)$, because the probability of a random walk (in this instance $(X_l(v), l\leq n)$) to stay below a barrier is exactly of the order $1/n$. This is the content of the ballot theorem that we now state precisely. 
The reader is referred to \cite{addario-reed1,addario-reed, bramson-ding-zeitouni2} for more details. 
The result quoted here is the version appearing in \cite{arguin-belius-harper}.
\begin{thm}
\label{thm: ballot}
Let $\left(S_{n}\right)_{n\ge0}$ be a
Gaussian random walk with increments of mean $0$ and variance $\sigma^2 > 0$, with $S_0=0$. Let $\delta>0$.
 There is a constant $C=C(\sigma,\delta)$ such that for all $B>0$, $b \le B - \delta$ and $n\ge1$
\begin{equation}\label{eq:ballot theorem ub}
  P\left[X_{n}\in\left(b,b+\delta\right)\mbox{ and }X_{k}\le B \mbox{ for }0<k<n\right] \le C\frac{(1+B)(1+B-b)}{n^{3/2}}.
\end{equation}
Moreover, if $\delta<1$,
\begin{equation}\label{eq:ballot theorem lb}
P\left[X_{n}\in\left(0,\delta\right)\mbox{ and }X_{k}\le 1 \mbox{ for }0<k<n\right]\geq   \frac{1}{Cn^{3/2}} \ .
\end{equation}
\end{thm}
With this new tool in hand, we are ready to prove the subleading order.
\begin{proof}[Proof of Theorem \ref{thm: second}]
\noindent {\it Upper bound:}
We want to show that for $\varepsilon>0$
$$
\lim_{n\to\infty}P(\max_{v\in V_n}X_v(n)\geq m_n(\varepsilon))=0\ .
$$
As mentioned previously, the idea is to first show that we can constrain the random walks to stay below a linear barrier.
More precisely, we prove that for $B=\log^2 n$, 
\begin{equation}
\label{eqn: multiscale markov}
\lim_{n\to\infty}P(\exists v\in V_n: X_v(n)\geq m_n(\varepsilon), X_l(v)> cl +B \ \text{ for some } l\leq n)=0\ .
\end{equation}
The proof is by a Markov's inequality at each scale. Indeed, we can bound the above by
$$
\sum_{l=1}^n P(\exists v: X_l(v)>cl+B)
\leq \sum_{l=1}^n 2^l \exp\left(-\frac{(cl+B)^2}{2\sigma^2 l}\right)
=n \exp (-Bc/\sigma^2)=o(1)\ ,
$$
where the first inequality is by the Gaussian estimate and the last by the choice of $B=\log^2 n$. 

To conclude the proof of the upper bound, consider the event
$$
\mathcal E^+(v)=\{X_v(n)\geq m_n(\varepsilon), X_l(v)\leq cl+B \ \forall l\leq n\}\ .
$$
By \eqref{eqn: multiscale markov}, it follows that
$$
P(\max_{v\in V_n}X_v(n)\geq m_n(\varepsilon))=P(\sum_{v\in V_n} 1_{\mathcal E^+(v)}\geq 1) + o(1)\ .
$$
Therefore, it remains to show that the first probability goes to $0$. By Markov's inequality, this is simply
$$
P(\sum_{v\in V_n} 1_{\mathcal E^+(v)}\geq 1)\leq 2^n P(\mathcal E^+(v))\ .
$$
To prove that this goes to $0$, it is useful to make a change of measure to make the end value $m_n(\varepsilon)$ typical. 
To this aim, consider the measure $Q$ defined from $P$ through the density
$$
\frac{dQ}{dP}=\prod_{\rm e} \frac{e^{\lambda Y_{\rm e}}}{e^{\varphi(\lambda)}},\qquad \varphi(\lambda)=\log E[e^{\lambda Y_{\rm e}}]=\frac{\lambda^2\sigma^2}{2}\ ,
$$
where the product is over all the edges in the binary tree. Note that by definition
$$
\varphi'(\lambda)=E_Q[Y_{\rm e}]=\lambda \sigma^2\ .
$$
In particular, to make $m_n(\varepsilon)$ typical, we choose
$$
\lambda=\frac{m_n(\varepsilon)}{\sigma^2n}\ .
$$
Under this change of measure, we can write the probability as
\begin{equation}
\label{eqn: Q}
\begin{aligned}
P(\mathcal E^+(v))&=e^{-(\lambda m_n(\varepsilon)-\varphi(\lambda))}\ E_Q[e^{-\lambda(X_v(n)-m_n(\varepsilon))}  1_{\mathcal E^+(v)}]\\
&\leq e^{-\frac{(m_n(\varepsilon))^2}{2\sigma^2n}}Q(\mathcal E^+(v))\ ,
\end{aligned}
\end{equation}
where we use the choice of $\lambda$ and the fact that $X_v(n)-m_n(\varepsilon)\geq 0$ on the event $\mathcal E^+(v)$. 
Since
$$
2^n\exp(-\frac{(m_n(\varepsilon))^2}{2\sigma^2n})=O(n^{-\varepsilon c/\sigma^2}) \ n^{3/2}, 
$$
the upper bound will be proved if
$$
 n^{3/2}Q(\mathcal E^+(v))=o(n^{\varepsilon c/\sigma^2})\ .
$$
But this follows from the ballot theorem (Theorem \ref{thm: ballot}) with $B=(\log n)^2$, since $(Y_l(v)-m_n(\varepsilon)/n, l\leq n)$ under $Q$ is a random walk of finite variance.\\

\noindent {\it Lower bound:}
As for the leading order, to get a matching lower bound, it is necessary to create independence by cutting off the increments at low scales. 
Of course, we cannot truncate as many scales. We choose
$$
r=\log \log n\ ,
$$
and consider the truncated walks $X_v(n)-X_v(r)$. We do not lose much by dropping these since, for $\varepsilon >0$,
\begin{equation}
\label{eqn: truncate}
\begin{aligned}
&P\big(\exists v : X_v(n)<m_n(-2\varepsilon), X_v(n)-X_r(v)\geq m_{n-r}(-\varepsilon)\big)\leq P(\exists v : X_v(r)< - 10 cr)\ ,
\end{aligned}
\end{equation}
where the inequality holds for $n$ large enough depending on $\varepsilon$.
By a union bound and the Gaussian estimate \eqref{eqn: gaussian estimate} (note that the variables are symmetric!),
$$
P(\exists v : X_v(r)< - 10 cr)\leq 2^r \exp(-100c^2r^2/(2\sigma^2))=2^{-99r } \to 0\ .
$$
In view of \eqref{eqn: truncate} and by redefining $\varepsilon>0$, we reduced the proof to showing
\begin{equation}
\label{eqn: to prove}
P\big(\exists v : X_v(n)-X_r(v)\geq m_{n-r}(-\varepsilon)\big)\to 1, \ n\to\infty. 
\end{equation}
For this purpose, consider 
\begin{equation}
\label{eqn: mu}
\mu= \frac{m_{n-r}(-\varepsilon)}{n-r} \qquad \frac{\mu^2}{2\sigma^2}=\log 2-\left(\frac{3}{2}-\frac{c}{\sigma^2}\varepsilon\right)\frac{\log (n-r)}{n-r}+o(n^{-1})\ .
\end{equation}
Denote the truncated walk and its recentering by
$$
X_v(r,n)=X_v(n)-X_r(v)\ , \qquad \overline X_v(r,n)=X_v(n)-X_r(v)-\mu(n-r)\ .
$$
With this notation, the event of exceedance with the barrier $B=1$ is
$$
\mathcal E^-(v)=\left\{v\in V_n:  \overline X_v(r,n)\in [0,\delta],  \overline X_v(l,n)\leq 1\ \forall l\geq r+1\right\}\ ,
$$
where $\delta>0$ is arbitrary. The relevant number of exceedances is
$$
\mathcal N^-(v)=\sum_{v\in V_n}1_{\mathcal E^-(v)}\ .
$$
By the Paley-Zygmund inequality, we have
$$
P(\mathcal N^-(v)\geq 1)\geq \frac{\big(E[\mathcal N^-(v)]\big)^2}{E[(\mathcal N^-(v))^2]}\ .
$$
To prove \eqref{eqn: to prove}, it remains to show that
\begin{equation}
\label{eqn: to prove 2}
 E[(\mathcal N^-(v))^2]=(1-o(1))\big(E[\mathcal N^-(v)]\big)^2\ .
\end{equation}
The second moment can be can be split as a sum over $v\wedge v'$. We get
\begin{equation}
\label{eqn: split}
\begin{aligned}
 E[(\mathcal N^-(v))^2]
 &=\sum_{v,v': \ v\wedge v'\leq r} P(\mathcal E^-(v)\cap \mathcal E^-(v'))\\
 &+\sum_{m=r+1}^n \sum_{v,v':  v\wedge v'=m}P(\mathcal E^-(v)\cap \mathcal E^-(v'))\ .
 \end{aligned}
\end{equation}
For the first term, note that $\mathcal E^-(v)$ is independent of $\mathcal E^-(v')$ because $v\wedge v'\leq r$ so the walks share no increments.
Thus, the first term is 
\begin{equation}
\label{eqn: first term}
\sum_{v,v': \ v\wedge v'\leq r} P(\mathcal E^-(v)\cap \mathcal E^-(v'))=(1-o(1))\big(E[\mathcal N^-(v)]\big)^2
\end{equation}
 since there are $2^{2n}(1-o(1))$ pairs $v,v'$ with $v\wedge v'\leq r$. 
We show that the second term is $o(1)\big(E[\mathcal N^-(v)]\big)^2$ to conclude the proof. 

Observe that, by proceeding as \eqref{eqn: Q} with the measure $Q$ (for $\lambda=\mu/\sigma^2$) and the Gaussian estimate \eqref{eqn: gaussian estimate}, the expected number of exceedances satisfy
\begin{equation}
\label{eqn: lower first}
\begin{aligned}
E[\mathcal N^-(v)]&\geq C 2^ne^{-\lambda\delta}e^{-\frac{\mu^2(n-r)}{2\sigma^2}}Q(\mathcal E^-(v))\\
& \geq C 2^{r}e^{-\lambda\delta}(n-r)^{\frac{c}{\sigma^2} \varepsilon} \ (n-r)^{3/2} Q(\mathcal E^-(v))\\
&\geq  C2^{r} (n-r)^{\frac{c}{\sigma^2}\varepsilon}e^{-\lambda\delta}
\end{aligned}
\end{equation}
for some constant $C>0$. Here we used the estimate on $\frac{\mu^2}{2\sigma^2}$ in \eqref{eqn: mu}. The last inequality is a consequence of the ballot theorem \ref{thm: ballot}.
For $v\wedge v'=m$, we decompose the probability $P(\mathcal E^-(v)\cap \mathcal E^-(v'))$ on the events $\{ \overline X_v(r,m)\in (-q,-q+1]\}$
for $q\geq 0$. Note that $X_v(m)=X_{v'}(m)$! We thus get
$$
\begin{aligned}
&P(\mathcal E^-(v)\cap \mathcal E^-(v'))\\
&=\sum_{q=0}^\infty P(\mathcal E^-(v)\cap \mathcal E^-(v')\cap \{ \overline X_v(m,n)\in (-q,-q+1]\})\\
&\leq \sum_{q=0}^\infty  P\big( \overline X_v(r,m)\in (-q,-q+1],  \overline X_v(r,l) \leq 1 \ \forall l=r+1,\dots, m\big) \\
&\hspace{0.25cm} \times\Big(P\big( \overline X_v(m,n)\geq q-1, \overline X_v(m,l) \leq 1+q \ \forall l=m+1,\dots, n\big)\big)^2
\end{aligned}
$$
where we used the bound on $X_v(m)$, and the square of the probabillity comes from the independence of increments after the branching scale $m$.
We evaluate these two probabilities.
Using the change of measure with $\lambda=\mu/\sigma^2$ and the bounds on  $\overline X_v(m,n)$, the first one is smaller than
$$
e^{\lambda q}e^{-\frac{\mu^2(n-r)}{2\sigma^2}} Q( \overline X_v(r,m)\in (-q,-q+1],  \overline X_v(r,l) \leq 1 \ \forall l=r+1,\dots, m\big) \ .
$$
By the ballot theorem, this is smaller than
$$
e^{\lambda q}e^{-\frac{\mu^2(n-r)}{2\sigma^2}} \frac{q+1}{(m-r)^{3/2}}\ .
$$
An identical treatment of the square of the second probability yields that it is smaller than
$$
e^{-2\lambda q}e^{-\frac{\mu^2(m-r)}{\sigma^2}}\frac{q+1}{(n-m)^{3}}\ .
$$
Putting all the above together, we finally obtained that 
$$
\sum_{m=r+1}^n \sum_{v,v':  v\wedge v'=m}P(\mathcal E^-(v)\cap \mathcal E^-(v'))
\leq \sum_{m=r+1}^n 2^{2n-m}\frac{e^{-\frac{\mu^2}{2\sigma^2}(2n-m-r))}}{(m-r)^{3/2}(n-m)^{3}}
 \sum_{q=0}^{\infty}(q+1)^2e^{-\lambda q}\ ,
$$
since there are at most $2^{2n-m}$ pairs with $v\wedge v'=m$. 
The last sum is finite. Moreover, using the estimate in \eqref{eqn: mu}, the first sum is smaller than
$$
C 2^r (n-r)^{2\varepsilon\frac{c}{\sigma^2}}\sum_{m=r+1}^n \frac{(n-r)^{\frac{3}{2}(2-\frac{m-r}{n-r})}}{(m-r)^{3/2}(n-m)^{3}}\ .
$$
It is not hard to show by an integral test that the series is bounded for any $n$. Thus
\begin{equation}
\label{eqn: high}
\sum_{m=r+1}^n \sum_{v,v':  v\wedge v'=m}P(\mathcal E^-(v)\cap \mathcal E^-(v'))
\leq C 2^r (n-r)^{2\varepsilon\frac{c}{\sigma^2}}\ .
\end{equation}
Using \eqref{eqn: high}, \eqref{eqn: lower first}, and \eqref{eqn: first term} in \eqref{eqn: split}, we conclude that
$$
P(\mathcal N^-(v)\geq 1)\
\geq \frac{1}{1+C2^r e^{2\lambda\delta}}
$$
which goes to $0$ by taking $n\to\infty$  ($r=\log\log n$) than $\delta\to 0$. This proves the theorem.
\end{proof}

\section{Universality classes of log-correlated fields}
\label{sect: universality}

The heuristics developed in Section \ref{sect: order} from branching random walks turns out to be applicable to a wide class of stochastic processes. 
It is a research topic of current interest to extend these techniques to other non-Gaussian log-correlated models with the purpose of proving results similar to Conjecture \ref{conj: informal}.
One interesting problem that we will not have time to allude to is the cover time of the discrete random walk on the two-dimensional torus, which perhaps surprisingly, can be estimated using such methods. We refer the reader to \cite{dembo-peres-rosen-zeitouni, belius-kistler} for results on this question.
We focus our attention here to two open problems: the large values of the Riemann zeta function and of the characteristic polynomials of random matrices. 
The behavior of the large values of these two models was conjectured by Fydorov, Hiary \& Keating  \cite{fyodorov-hiary-keating, fyodorov-keating} to mimic that of a Gaussian log-correlated model. We shall explain here this connection in terms of branching random walks.

\subsection{Maximum of the Riemann zeta function on an interval}
Let $s\in \mathbb C$. If  $\Re \ s>1/2$, the Riemann zeta function is defined by
\begin{equation}
\zeta(s)=\sum_{n=1}^\infty \frac{1}{n^s}=\prod_{p \text{ primes}} (1-p^{-s})^{-1}\ .
\end{equation}
This definition can be analytically continued using the functional equation 
$$
\zeta(s)= \chi(s) \zeta(1-s)\ ,\qquad \chi(s) =2^s \pi^{s-1}\sin \left(\frac{\pi}{2}s\right)\Gamma(1-s)\ ,
$$
to the whole complex plane with a pole at $s=1$. 
The distribution of the primes is famously linked to information on the zeros of the function.
There are trivial zeros (coming from the functional equation) at the negative even integers.
The Riemann Hypothesis states that all other zeros lie on the critical line $\Re\  s=1/2$ suggested by the symmetry of the functional equation.
See Figure \ref{fig: zeta} for the plot of the modulus of $\zeta$ on the critical line.
\begin{figure}[h]
\includegraphics[height=6cm]{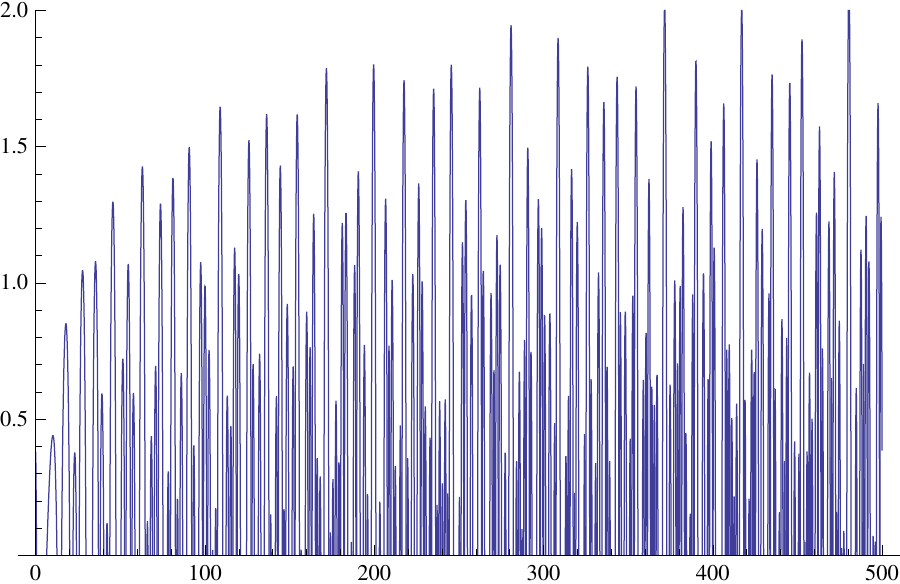}
\caption{The function $\log|\zeta(s)|$ for $s=1/2+it$ for $t\in[0,500]$.}
\label{fig: zeta}
\end{figure}

We shall be interested instead in the large values of the modulus of $\zeta$ in an interval. 
In other words, what can be said about the local maxima of zeta ?
The global behavior is much more intricate, see the Lindel\"of's hypothesis.

Fyodorov, Hiary and \& Keating based on numerics and connection with random matrices made an astonishing prediction for the local maximum of $\zeta$ on an interval, say $[0,1]$, of the critical line.
\begin{conj}[Fyodorov--Hiary--Keating \cite{fyodorov-hiary-keating, fyodorov-keating}]
\label{conj: FHK zeta}
For $\tau$ sampled uniformly from $[0,T]$,
\begin{equation}
\label{eqn: FHK}
\lim_{T\to \infty}\frac{\max_{h\in [0,1]} \log |\zeta(1/2+i(\tau +h))|-\log\log T}{\log\log\log T}= -\frac{3}{4} \text{ in probability.}
\end{equation}
\end{conj}
In fact, their conjecture is even more precise, suggesting an explicit distribution for the fluctuation. 
Roughly speaking, the above means that the maximum in a typical interval of the critical line is
$$
\max_{h\in [0,1]} \log |\zeta(1/2+i(\tau +h))|\approx\log\log T- \frac{3}{4}\log\log\log T +O(1)\ .
$$
The connection with Conjecture \ref{conj: informal} is made by taking $\log T=2^n$ and $\sigma^2=\frac{\log 2}{2}$.
Observe that it is a perfect match!
As we will see, the relation with log-correlated models goes beyond the numerological curiosity.
It also proposes a starting point to prove the conjecture. In fact, the conjecture can be proved to the subleading order for a random model of zeta, cf. Theorem \ref{thm: random zeta}.

Suppose first the Euler product of the Riemann zeta function extended to the critical line.
By expanding the logarithm, we would have
\begin{equation}
\label{eqn: Euler product taylor}
\log|\zeta(1/2+it) |= -\Re \sum_{p} \log(1-p^{-(1/2+it)}) = \sum_{k=1}^{\infty} \frac{1}{k} \sum_{p} \Re\frac{1}{p^{k(1/2+it)}} = \sum_{p} \frac{\Re \ p^{-it}}{p^{1/2}} + O(1) \ .
\end{equation}
This expansion can be made partially rigorous if the sum over primes is cut off at $T$, see Proposition 1 of \cite{harper} based on  \cite{sound}.
Now replace $t$ by a random $\tau+h$ where $h\in[0,1]$ and $\tau$ is sampled uniformly from $[0,T]$. 
The randomness entirely comes from the random variables $(p^{-i\tau}, p \text{ primes})$.
It is an easy exercise by moment computations to verify that, since the values of $\log p$ are linearly independent for distinct primes,  the process $(p^{-i\tau} , \text{ $p$ primes})$ converges in the sense of finite-dimensional distribution as $T\to\infty$ to independent random variables distributed uniformly on the unit circle. This suggests the following model for zeta: let $(U_p,p \text{ primes})$ be IID~ uniform random variables on the unit circle index by the primes and take
 \begin{equation}
 \label{eqn: process}
\left(\sum_{p\leq T}\frac{\Re (U_p p^{-ih})}{p^{1/2}}, h\in [0,1]\right)\ .
\end{equation}

For this model, we can prove the conjecture up to subleading order.
\begin{thm}[\cite{arguin-belius-harper}]
\label{thm: random zeta}
Let $(U_p,p \text{ primes})$ be independent random variables distributed uniformly on the unit circle.
Then
\begin{equation}
\label{eq: main}
\lim_{T\to\infty}\frac{\max_{h\in [0,1]} \sum_{p\leq T}\frac{\Re (U_p p^{-ih})}{p^{1/2}} -\log\log T}{\log\log\log T}=-\frac{3}{4} \text{ in probability.}
\end{equation}
\end{thm}
The proof is based on finding an approximate branching random walk embedded in zeta.
What is this approximate branching random walk ? In particular, what is the mutliscale decomposition and does it exhibit the usual self-similarity and dichotomy of scales of log-correlated models ?
It is tempting to think of the decomposition $\sum_{p} \frac{\Re \ p^{-it}}{p^{1/2}}$ as the right multiscale decomposition since the variables $\Re \ p^{-it}$ decorrelate in the limit. However, for the scales to be self-similar, it is necessary to consider a coarse-graining of the sum. 
For the sake of clarity, we take $\log T=2^n$.
For the random model \eqref{eqn: process}, the increment at scale $l$ for $l=1,\dots, n$ is 
\begin{equation}
\label{eqn: Y}
Y_h(l)=\sum_{2^{l-1}< \log p \leq 2^l} \frac{\Re (U_p p^{-ih})}{p^{1/2}}  \ , h\in[0,1].
\end{equation}
The multiscale decomposition is therefore
$$
 \sum_{p\leq T}\frac{\Re (U_p p^{-ih})}{p^{1/2}}=\sum_{l=1}^n Y_h(l)\ .
$$
For the random model, the increments are independent by construction. For zeta, they are only independent in the limit which is a substantial difficulty.
We verify that the $Y(l)$'s are self-similar and exhibit a dichotomy.
The variance is easily calculated for the random model
$$
E[(Y_h(k)^2]=\sum_{2^{k-1}< \log p \leq 2^k} \frac{1}{p}E[(\Re U_p)^2]=\frac{1}{2}\sum_{2^{k-1}< \log p \leq 2^k} \frac{1}{p}
$$
where we used the fact that $E[U_p^2]=0$ and $E[U_p\overline{U}_p]=0$ where $\overline z$ is the complex conjugate of $z$.
The sum over primes can be evaluated using the {\it Prime Number Theorem}, see e.g. \cite{montgomery-vaughan-multiplicative-nt}, which states that
\begin{equation}\label{eq: PNT}
 \#\{p \le x: p \mbox{ prime}\} = \int_2^{x} \frac{1}{\log u}du + O( x e^{-c \sqrt{ \log x }} ).
\end{equation}
We get
$$
E[(Y_h(l)^2]=\frac{\log 2}{2} + O(e^{-c\sqrt{2^l}})\ .
$$
Thus the variances of the increments are approximately equal. Moreover, the parameter $\sigma^2$ should be taken to be $\log 2/2$.
This motivates the above coarse-graining to get {\it self-similarity}. 
For the {\it dichotomy}, we need the equivalent of the branching scale.
For $h,h'\in[0,1]$, take
\begin{equation}
\label{eqn: branching}
h\wedge h'= \log_2 |h-h'|^{-1}\ .
\end{equation}
The covariance between increments at scale $l$ is
$$
E[Y_h(k)Y_{h'}(k)]=\frac{1}{2}\sum_{2^{k-1}< \log p \leq 2^k} \frac{\cos(|h-h'|\log p)}{p} \ .
$$
Again, the sum can be evaluated using \eqref{eq: PNT}. This time, the oscillating nature of cosine will produce a dichotomy depending on the branching scale between $h$ and $h'$. In fact, it is not hard to show that, see Lemma 2.1 in \cite{arguin-belius-harper},
\begin{equation}
\label{eqn: correlation estimates}
E [Y_{k}(h) Y_{k}(h')] = 
\begin{cases}
\frac{\log 2}{2}+ O\left( 2^{-2(h\wedge h'-k)} \right) &\text{ if $k\leq h\wedge h'$,}\\
O\left( 2^{-(k- h\wedge h')} \right)  &\text{ if $k> h\wedge h'$.}
\end{cases}
\end{equation}
The dichotomy is not as clean as for BRW. On the other hand, the increments couple and decouple exponentially fast in the scales.
This turns out to be sufficient for the purpose of the subleading order. 

To prove Theorem \ref{thm: random zeta}, there are additional difficulties. 
On one hand, the process is continuous and not discrete like BRW. 
Therefore, an argument is needed to show that a discrete set of $2^n$ points capture the order of the maximum.
This is based on a chaining argument. 
Finally, the process here is not Gaussian so the estimate \eqref{eqn: gaussian estimate} cannot be used.
It is replaced by precise large deviation principles obtained from estimates on the exponential moments of the sum.

Of course, there are tremendous technical obstacles to extend the method of proof of Theorem \ref{thm: random zeta} to prove Conjecture \ref{conj: FHK zeta} on the actual zeta function. For one, large deviation estimates are much harder to get. Moreover, the contribution to $\zeta$ of primes larger than $T$ needs to be addressed. Despite these hurdles, the branching random walk heuristic gives a promising path to prove at least the leading order of the conjecture. 

\subsection{Maximum of the characteristic polynomial of random unitary matrices}
\label{sect: cue}

In this final section, we return to the problem we alluded to in Section \ref{sect: stat}.
Let $\rm U_N$ be a random $N\times N$ unitary matrix sampled uniformly from the unitary group.
This is often referred to as the {\it Circular Unitary Ensemble} (CUE).
We consider the modulus of the characteristic polynomial on the unit circle 
$$
\left|{\rm P}_{\rm U_N}(\theta)\right|=\left|\det(e^{i\theta}-{\rm U}_N)\right|=\prod_{j=1}^{2^{10}} |e^{i\theta}- e^{i\lambda_j}|\ .
$$
where the eigenvalues of $U_N$ are denoted by $(\lambda_j, j\leq N)$ and lie on the unit circle.
A realization of this process was given for $N=2^{10}$ in Figure \ref{fig: unitary}.
The prediction for the order of the maximum of this process is
\begin{conj}[Fyodorov--Hiary--Keating \cite{fyodorov-hiary-keating, fyodorov-keating}]
\label{conj: fhk unitary}
For $N\in \mathbb N$, let $\rm U_N$ be a random matrix sampled uniformly from the group of $N\times N$ unitary matrices.
Then
\begin{equation}
\lim_{N\to\infty} \frac{\max_{\theta\in[0,2\pi]} \log |{\rm P}_{\rm U_N}(\theta)|-\log N}{\log\log N}=-\frac{3}{4} \ \text{ in probability.}
\end{equation}
\end{conj}
In other words, it it expected that
$$
\max_{\theta\in[0,2\pi]} \log |{\rm P}_{\rm U_N}(\theta)|\approx \log N -\frac{3}{4}\log\log N+O(1)\ .
$$
The conjecture is well motivated by numerics and precise computations by Fyodorov \& Keating of moments of the partition function of the models.
They infer from the expression for the moments the order of magnitude of the moments as well as a prediction for the fluctuations of the maximum assuming that the system undergoes a {\it freezing transition} similar to the random energy model. 
Since there is strong empirical evidence that the characteristic polynomial of CUE is a good model for the Riemann zeta function locally, 
the authors use in part Conjecture \ref{conj: fhk unitary} to motivate Conjecture \ref{conj: FHK zeta}.

As for the Riemann zeta function, the conjecture is exactly what would be expected for the maximum of a log-correlated model with $N=2^n$ random variables with variance $\sigma^2 n=\frac{\log 2}{2} n$, cf. Conjecture \ref{conj: informal}. 
It turns out that for a given $\theta$, it was shown in \cite{keating-snaith} that $ \log |{\rm P}_{\rm U_N}(\theta)|$ normalized by $\sqrt{\frac{1}{2}\log N}$ converges in distribution to a standard Gaussian. Therefore, the choice of of $\sigma^2$ is already consistent with this result.
Kistler's multiscale refinement of the second moment method as described in Section \ref{sect: order} can be adapted to prove the leading order of the conjecture.
\begin{thm}[\cite{arguin-belius-bourgade}]
\label{thm: cue}
For $N\in \mathbb N$, let ${\rm U}_N$ be a random matrix sampled uniformly from the group of $N\times N$ unitary matrices.
Write ${\rm P}_N(\theta)$, $\theta\in[0,2\pi]$, for its characteristic polynomial on the unit circle.
Then
\begin{equation}
\lim_{N\to\infty}\frac{\underset{\theta\in[0,2\pi]}{\max} \log |{\rm P}_N(\theta)|}{\log N}=1 \qquad \text{in probability.}
\end{equation}
\end{thm}
In this section, we explain the connection to an approximate branching random walk which illustrates the approach to prove the theorem and the general conjecture.

To see what plays the role of the {\it multiscale decomposition}, write the characteristic polynomial as
$$
\log|{\rm P}_{{\rm U}_N}(\theta)| = \sum_{j=1}^N \log| 1 - e^{i(\lambda_j-\theta)} |.
$$
By expanding the logarithm, we get
\begin{equation}
\label{eqn: fourier}
\log|{\rm P}_{{\rm U}_N}(\theta)| = \sum_{j=1}^N \sum_{k=1}^\infty -\frac{\Re (e^{i k(\lambda_j-\theta)})}{j}
= \sum_{k=1}^\infty -\frac{\Re (e^{-i k \theta}\Tr {\rm U}_N^k)}{k} 
\end{equation}
where $\Tr$ stands for the trace.
It turns out that an integrable periodic function always has a pointwise convergent Fourier series wherever it is differentiable, so the above series makes sense for $\theta$ away from the eigenvalues.
This already seems a good candidate for the {\it multiscale decomposition}. 
However, as it was the case for zeta, it will be necessary to group the traces of powers to get the equivalent of increments at each scale. 
In other words, the traces of the powers play the role of the primes.

To see this, we will need the seminal result of Diaconis \& Shahshahani \cite{diaconis-shah}, see also \cite{diaconis-evans}, who proved that
\begin{equation}
\label{eqn: diaconis-shah}
\E\left[\Tr {\rm U}_N^j\ \overline{\Tr {\rm U}_N^k}\right]=\delta_{kj} \min(k,N).
\end{equation}
where $\mathbb{E}$ stands for the expectation under the uniform measure on the unitary group.
It is also easy to see by rotation invariance of the uniform measure that $\mathbb{E}\left[\Tr {\rm U}_N^j \Tr {\rm U}_N^k\right]=0$.
This shows that that traces of powers are {\it uncorrelated}. However, and this is one of the main issue, they are not independent for fixed $N$.
The formula \eqref{eqn: diaconis-shah} is useful to our heuristic in many ways.
First, an easy calculation shows that the variance of the powers less than $N$ is
$$
E\left[ \left(\sum_{k=1}^N -\frac{\Re (e^{-i k\theta}\Tr {\rm U}_N^k)}{k}\right)^2\right] = \frac{1}{2}\sum_{k\leq N}\frac{1}{k}= \frac{1}{2} \log N + O(1).
$$
In particular, if we take $N=2^n$ to emphasize the analogy with BRW, it leads us to define the {\it increment at scale $l$} for $l=1,\dots, n$ as
$$
Y_\theta(l)=\sum_{2^{l-1}<k\leq 2^l} -\frac{\Re (e^{-i k\theta}\Tr {\rm U}_N^k)}{k}\ .
$$
This gives
$$
\E[(Y_\theta(l))^2]=\frac{\log 2}{2} +o(1)\ ,
$$
where $o(1)$ is summable in $l$.
From \eqref{eqn: fourier}, the candidate for the multiscale decomposition is
\begin{equation}
\label{eqn: multi cue}
\log|{\rm P}_{{\rm U}_N}(\theta)| =\sum_{l=1}^n Y_\theta(l) + Z_\theta(n)\ ,
\end{equation}
where $ Z_\theta(n)$ stands for the sum of powers greater than $n$. 
Moreover, equation \eqref{eqn: diaconis-shah} suggests that the contribution of $Z_\theta(n)$ to $\log |{\rm P}_N(\theta)|$ in \eqref{eqn: fourier} should be of order $1$ since
$\sum_{j> N} \frac{N}{j^2}=O(1)$.
We stress that, in this decomposition, the increments are uncorrelated though not independent. 
They are however asymptotically independent Gaussians, cf. \cite{diaconis-shah}. Finally, they are almost {\it self-similar} in that regard since their variances are almost identical. 

It remains to show the {\it dichotomy of scales} to complete the connection with BRW.
A simple calculation using \eqref{eqn: diaconis-shah} gives that the covariance between increments for $\theta,\theta'\in[0,2\pi]$ is
\begin{equation}
\label{eqn: branching2}
\mathbb{E}[Y_\theta(l)Y_{\theta'}(l)]=\sum_{2^{l-1}<k\leq 2^l}\frac{\cos(j\|\theta-\theta'\|)}{2k}\ ,
\end{equation}
where $\|\theta-\theta'\|$ stands for the periodic distance on $[0,2\pi]$. 
As it was the case for the Riemann zeta function, the presence of the cosine is responsible for the dichotomy.
To see this, define the {\it branching scale} 
\begin{equation}
\label{eqn: branching scale}
\theta\wedge \theta'=-\log_2 \|\theta-\theta'\|\ .
\end{equation}
For $j$ such that $j \|h-h'\|$ is small the cosine is essentially $1$, and for $j$ such that $j \|\theta-\theta'\|$
is large the oscillation of the cosine leads to cancellation. 
This can be proved by Taylor expansion of $\cos$ in the first case and summation by parts in the other.
This argument shows
\begin{equation}
\label{eqn: branching3}
\mathbb{E}[Y_\theta(l)Y_{\theta'}(l)]=
\begin{cases}
\frac{1}{2} +O(2^{l-\theta\wedge \theta'}) \ &\text{if $l\leq \theta\wedge \theta'$,}\\
O(2^{-2(l-\theta\wedge \theta')}) 			\ &\text{if $l> \theta\wedge \theta'$.}\\
\end{cases}
\end{equation}
In particular, this shows that the sum of traces of powers less than $N$ are {\it log-correlated}:
$$
\mathbb{E}\left[\left(\sum_{l=1}^{n} Y_\theta(l)\right)\left(\sum_{l=1}^{n} Y_{\theta'}(l)\right)\right]
=\frac{\theta \wedge \theta'}{2}+O(1)=-\frac{1}{2}\log \|\theta-\theta'\|+O(1)\ .
$$

There are substantial difficulties to implement the method of Section \ref{sect: order} to prove Theorem \ref{thm: cue} and more generally Conjecture \ref{conj: fhk unitary}.
First, it is necessary to control the contribution of the high powers, that is $Z_\theta(n)$ in \eqref{eqn: multi cue}.
This turns out to be technically very challenging from a random matrix standpoint. Second, as opposed to BRW, the increments are not independent.
This makes it much harder to get good large deviation estimates on the sum of the increments. Finally, since we are dealing with a continuous process on $[0,2\pi]$ with singularities at the eigenvalues, one needs to justify that taking the maximum on $N=2^n$ discrete points is enough to capture the order of the maximum.

\bibliographystyle{plain}
\bibliography{bib_cirm}

\begin{bibdiv}
\begin{biblist}

\bib{addario-reed}{article}{
      author={Addario-Berry, L.},
      author={Reed, B.},
       title={Minima in branching random walks},
        date={2009},
        ISSN={0091-1798},
     journal={Ann. Probab.},
      volume={37},
      number={3},
       pages={1044\ndash 1079},
         url={http://dx.doi.org/10.1214/08-AOP428},
      review={\MR{2537549 (2011b:60338)}},
}

\bib{addario-reed1}{incollection}{
      author={Addario-Berry, L.},
      author={Reed, B.~A.},
       title={Ballot theorems, old and new},
        date={2008},
   booktitle={Horizons of combinatorics},
      series={Bolyai Soc. Math. Stud.},
      volume={17},
   publisher={Springer, Berlin},
       pages={9\ndash 35},
}

\bib{aidekon}{article}{
      author={A{\"{\i}}d{\'e}kon, E.},
       title={Convergence in law of the minimum of a branching random walk},
        date={2013},
        ISSN={0091-1798},
     journal={Ann. Probab.},
      volume={41},
      number={3A},
       pages={1362\ndash 1426},
         url={http://dx.doi.org/10.1214/12-AOP750},
      review={\MR{3098680}},
}

\bib{abbs}{article}{
      author={A{\"{\i}}d{\'e}kon, E.},
      author={Berestycki, J.},
      author={Brunet, {\'E}.},
      author={Shi, Z.},
       title={Branching {B}rownian motion seen from its tip},
        date={2013},
        ISSN={0178-8051},
     journal={Probab. Theory Related Fields},
      volume={157},
      number={1-2},
       pages={405\ndash 451},
         url={http://dx.doi.org/10.1007/s00440-012-0461-0},
      review={\MR{3101852}},
}

\bib{aidekon-shi}{article}{
      author={A{\"{\i}}d{\'e}kon, E.},
      author={Shi, Z.},
       title={Weak convergence for the minimal position in a branching random
  walk: a simple proof},
        date={2010},
        ISSN={0031-5303},
     journal={Period. Math. Hungar.},
      volume={61},
      number={1-2},
       pages={43\ndash 54},
         url={http://dx.doi.org/10.1007/s10998-010-3043-x},
      review={\MR{2728431 (2011g:60153)}},
}

\bib{arguin-belius-bourgade}{article}{
      author={Arguin, L.-P.},
      author={Belius, D.},
      author={Bourgade, P.},
       title={Maximum of the characteristic polynomial of random unitary
  matrices},
        date={2015},
     journal={Preprint arXiv:1511.07399},
         url={http://arxiv.org/abs/1511.07399},
}

\bib{arguin-belius-harper}{article}{
      author={Arguin, L.-P.},
      author={Belius, D.},
      author={Harper, A.~J.},
       title={Maxima of a randomized riemann zeta function, and branching
  random walks},
        date={2015},
     journal={Preprint arXiv:1506.00629},
         url={http://arxiv.org/abs/1506.00629},
}

\bib{arguin-bovier-kistler_genealogy}{article}{
      author={Arguin, L.-P.},
      author={Bovier, A.},
      author={Kistler, N.},
       title={Genealogy of extremal particles of branching {B}rownian motion},
        date={2011},
        ISSN={0010-3640},
     journal={Comm. Pure Appl. Math.},
      volume={64},
      number={12},
       pages={1647\ndash 1676},
         url={http://dx.doi.org/10.1002/cpa.20387},
      review={\MR{2838339 (2012g:60269)}},
}

\bib{arguin-bovier-kistler_poisson}{article}{
      author={Arguin, L.-P.},
      author={Bovier, A.},
      author={Kistler, N.},
       title={Poissonian statistics in the extremal process of branching
  {B}rownian motion},
        date={2012},
        ISSN={1050-5164},
     journal={Ann. Appl. Probab.},
      volume={22},
      number={4},
       pages={1693\ndash 1711},
         url={http://dx.doi.org/10.1214/11-AAP809},
      review={\MR{2985174}},
}

\bib{arguin-bovier-kistler_extremal}{article}{
      author={Arguin, L.-P.},
      author={Bovier, A.},
      author={Kistler, N.},
       title={The extremal process of branching {B}rownian motion},
        date={2013},
        ISSN={0178-8051},
     journal={Probab. Theory Related Fields},
      volume={157},
      number={3-4},
       pages={535\ndash 574},
         url={http://dx.doi.org/10.1007/s00440-012-0464-x},
      review={\MR{3129797}},
}

\bib{arguin-ouimet}{article}{
      author={Arguin, L.-P.},
      author={Ouimet, F.},
       title={Extremes of the two-dimensional gaussian free field with
  scale-dependent variance},
        date={2015},
     journal={Preprint arXiv:1508.06253},
}

\bib{arguin-zindy}{article}{
      author={Arguin, L.-P.},
      author={Zindy, O.},
       title={Poisson-{D}irichlet statistics for the extremes of a
  log-correlated {G}aussian field},
        date={2014},
        ISSN={1050-5164},
     journal={Ann. Appl. Probab.},
      volume={24},
      number={4},
       pages={1446\ndash 1481},
         url={http://dx.doi.org/10.1214/13-AAP952},
      review={\MR{3211001}},
}

\bib{arguin-zindy2}{article}{
      author={Arguin, L.-P.},
      author={Zindy, O.},
       title={Poisson-{D}irichlet statistics for the extremes of the
  two-dimensional discrete {G}aussian free field},
        date={2015},
        ISSN={1083-6489},
     journal={Electron. J. Probab.},
      volume={20},
       pages={no. 59, 19},
      review={\MR{3354619}},
}

\bib{bachmann}{article}{
      author={Bachmann, M.},
       title={Limit theorems for the minimal position in a branching random
  walk with independent logconcave displacements},
        date={2000},
        ISSN={0001-8678},
     journal={Adv. in Appl. Probab.},
      volume={32},
      number={1},
       pages={159\ndash 176},
         url={http://dx.doi.org/10.1239/aap/1013540028},
      review={\MR{1765165 (2001m:60189)}},
}

\bib{belius-kistler}{article}{
      author={Belius, D.},
      author={Kistler, N.},
       title={The subleading order of two dimensional cover times},
        date={2014},
     journal={Preprint arXiv:1405.0888},
         url={http://arxiv.org/abs/1405.0888},
}

\bib{Biggins1976}{article}{
      author={Biggins, J.~D.},
       title={The first- and last-birth problems for a multitype age-dependent
  branching process},
        date={1976},
     journal={Advances in Appl. Probability},
      volume={8},
      number={3},
       pages={446\ndash 459},
}

\bib{biskup-louidor}{article}{
      author={Biskup, M.},
      author={Louidor, O.},
       title={Extreme local extrema of two-dimensional discrete gaussian free
  field},
        date={2013},
     journal={Preprint arXiv:1306.2602},
         url={http://arxiv.org/abs/1306.2602},
}

\bib{bolthausen-deuschel-giacomin}{article}{
      author={Bolthausen, E.},
      author={Deuschel, J.-D.},
      author={Giacomin, G.},
       title={Entropic repulsion and the maximum of the two-dimensional
  harmonic crystal},
        date={2001},
        ISSN={0091-1798},
     journal={Ann. Probab.},
      volume={29},
      number={4},
       pages={1670\ndash 1692},
         url={http://dx.doi.org/10.1214/aop/1015345767},
      review={\MR{1880237 (2003a:82028)}},
}

\bib{bovier_extremes}{article}{
      author={Bovier, A.},
       title={Extremes},
        date={2006},
     journal={Online lecture notes},
  url={http://wt.iam.uni-bonn.de/fileadmin/WT/Inhalt/people/Anton_Bovier/lecture-notes/extreme.pdf},
        note={available at
  \url{http://wt.iam.uni-bonn.de/fileadmin/WT/Inhalt/people/Anton_Bovier/lecture-notes/extreme.pdf}},
}

\bib{bovier_bbm}{incollection}{
      author={Bovier, A.},
       title={From spin glasses to branching brownian motion -- and back?},
        date={2015},
   booktitle={Random walks, random fields, and disordered systems},
      series={Lecture Notes in Math.},
      volume={2144},
   publisher={Springer, Berlin},
       pages={1\ndash 64},
}

\bib{BovierHartung14}{article}{
      author={Bovier, A.},
      author={Hartung, L.},
       title={The extremal process of two-speed branching brownian motion},
        date={2014},
     journal={Elect. J. Probab.},
      volume={19},
      number={18},
       pages={1\ndash 28},
}

\bib{BovierHartung2015}{article}{
      author={Bovier, A.},
      author={Hartung, L.},
       title={Variable speed branching {B}rownian motion 1. {E}xtremal
  processes in the weak correlation regime},
        date={2015},
        ISSN={1980-0436},
     journal={ALEA Lat. Am. J. Probab. Math. Stat.},
      volume={12},
      number={1},
       pages={261\ndash 291},
      review={\MR{3351476}},
}

\bib{bramson-ding-zeitouni}{article}{
      author={Bramson, M.},
      author={Ding, J.},
      author={Zeitouni, O.},
       title={Convergence in law of the maximum of the two-dimensional discrete
  gaussian free field},
        date={2013},
     journal={Preprint arXiv:1301.6669},
         url={http://arxiv.org/abs/1301.6669},
}

\bib{bramson-ding-zeitouni2}{article}{
      author={Bramson, M.},
      author={Ding, J.},
      author={Zeitouni, O.},
       title={Convergence in law of the maximum of nonlattice branching random
  walk},
        date={2014},
     journal={Preprint arXiv: 1404.3423},
         url={http://arxiv.org/abs/1404.3423},
}

\bib{bramson}{article}{
      author={Bramson, M.~D.},
       title={Maximal displacement of branching {B}rownian motion},
        date={1978},
        ISSN={0010-3640},
     journal={Comm. Pure Appl. Math.},
      volume={31},
      number={5},
       pages={531\ndash 581},
      review={\MR{0494541 (58 \#13382)}},
}

\bib{carpentier-ledoussal}{article}{
      author={Carpentier, D.},
      author={Le~Doussal, P.},
       title={Glass transition of a particle in a random potential, front
  selection in nonlinear renormalization group, and entropic phenomena in
  liouville and sinh-gordon models},
        date={2001Jan},
     journal={Phys. Rev. E},
      volume={63},
       pages={026110},
         url={http://link.aps.org/doi/10.1103/PhysRevE.63.026110},
}

\bib{daviaud}{article}{
      author={Daviaud, O.},
       title={Extremes of the discrete two-dimensional {G}aussian free field},
        date={2006},
     journal={Ann. Probab.},
      volume={34},
      number={3},
       pages={962\ndash 986},
}

\bib{dehaan-ferreira}{book}{
      author={de~Haan, L.},
      author={Ferreira, A.},
       title={Extreme value theory},
      series={Springer Series in Operations Research and Financial
  Engineering},
   publisher={Springer, New York},
        date={2006},
        ISBN={978-0-387-23946-0; 0-387-23946-4},
        note={An introduction},
      review={\MR{2234156 (2007g:62008)}},
}

\bib{dembo-peres-rosen-zeitouni}{article}{
      author={Dembo, A.},
      author={Peres, Y.},
      author={Rosen, J.},
      author={Zeitouni, O.},
       title={Cover times for {B}rownian motion and random walks in two
  dimensions},
        date={2004},
        ISSN={0003-486X},
     journal={Ann. of Math. (2)},
      volume={160},
      number={2},
       pages={433\ndash 464},
         url={http://dx.doi.org/10.4007/annals.2004.160.433},
      review={\MR{2123929 (2005k:60261)}},
}

\bib{derrida}{article}{
      author={Derrida, B.},
       title={Random-energy model: an exactly solvable model of disordered
  systems},
        date={1981},
        ISSN={0163-1829},
     journal={Phys. Rev. B (3)},
      volume={24},
      number={5},
       pages={2613\ndash 2626},
      review={\MR{627810 (83a:82018)}},
}

\bib{derrida_grem}{article}{
      author={Derrida, B.},
       title={A generalisation of the random energy model that includes
  correlations between the energies},
        date={1985},
     journal={J. Phys. Lett.},
      volume={46},
      number={3},
       pages={401\ndash 407},
}

\bib{derrida-spohn}{article}{
      author={Derrida, B.},
      author={Spohn, H.},
       title={Polymers on disordered trees, spin glasses, and traveling waves},
        date={1988},
        ISSN={0022-4715},
     journal={J. Statist. Phys.},
      volume={51},
      number={5-6},
       pages={817\ndash 840},
         url={http://dx.doi.org/10.1007/BF01014886},
        note={New directions in statistical mechanics (Santa Barbara, CA,
  1987)},
      review={\MR{971033 (90i:82045)}},
}

\bib{diaconis-evans}{article}{
      author={Diaconis, P.},
      author={Evans, S.~N.},
       title={Linear functionals of eigenvalues of random matrices},
        date={2001},
        ISSN={0002-9947},
     journal={Trans. Amer. Math. Soc.},
      volume={353},
      number={7},
       pages={2615\ndash 2633},
         url={http://dx.doi.org/10.1090/S0002-9947-01-02800-8},
      review={\MR{1828463 (2002d:60003)}},
}

\bib{diaconis-shah}{article}{
      author={Diaconis, P.},
      author={Shahshahani, M.},
       title={On the eigenvalues of random matrices},
        date={1994},
        ISSN={0021-9002},
     journal={J. Appl. Probab.},
      volume={31A},
       pages={49\ndash 62},
        note={Studies in applied probability},
      review={\MR{1274717 (95m:60011)}},
}

\bib{ding-roy-zeitouni}{article}{
      author={Ding, J.},
      author={Roy, O., R.and~Zeitouni},
       title={Convergence of the centered maximum of log-correlated gaussian
  fields},
        date={2015},
     journal={Preprint arXiv:1503.04588},
         url={http://arxiv.org/abs/arXiv:1503.04588},
}

\bib{FangZeitouni2012}{article}{
      author={Fang, M.},
      author={Zeitouni, O.},
       title={Branching random walks in time-inhomogeneous environments},
        date={2012},
     journal={Electron. J. Probab.},
      volume={17},
      number={67},
       pages={18},
}

\bib{FangZeitouni2012JSP}{article}{
      author={Fang, M.},
      author={Zeitouni, O.},
       title={Slowdown for time inhomogeneous branching {B}rownian motion},
        date={2012},
        ISSN={0022-4715},
     journal={J. Stat. Phys.},
      volume={149},
      number={1},
       pages={1\ndash 9},
         url={http://dx.doi.org/10.1007/s10955-012-0581-z},
      review={\MR{2981635}},
}

\bib{fyodorov-bouchaud}{article}{
      author={Fyodorov, Y.~V.},
      author={Bouchaud, J.-P.},
       title={Freezing and extreme-value statistics in a random energy model
  with logarithmically correlated potential},
        date={2008},
        ISSN={1751-8113},
     journal={J. Phys. A},
      volume={41},
      number={37},
       pages={372001, 12},
         url={http://dx.doi.org/10.1088/1751-8113/41/37/372001},
      review={\MR{2430565 (2010g:82036)}},
}

\bib{fyodorov-hiary-keating}{article}{
      author={Fyodorov, Y.~V.},
      author={Hiary, G.~A.},
      author={Keating, J.~P.},
       title={Freezing transition, characteristic polynomials of random
  matrices, and the riemann zeta function},
        date={2012Apr},
     journal={Phys. Rev. Lett.},
      volume={108},
       pages={170601},
         url={http://link.aps.org/doi/10.1103/PhysRevLett.108.170601},
}

\bib{fyodorov-keating}{article}{
      author={Fyodorov, Y.~V.},
      author={Keating, J.~P.},
       title={Freezing transitions and extreme values: random matrix theory,
  and disordered landscapes},
        date={2014},
        ISSN={1364-503X},
     journal={Philos. Trans. R. Soc. Lond. Ser. A Math. Phys. Eng. Sci.},
      volume={372},
      number={2007},
       pages={20120503, 32},
         url={http://dx.doi.org/10.1098/rsta.2012.0503},
      review={\MR{3151088}},
}

\bib{fyodorov-ledoussal-rosso}{article}{
      author={Fyodorov, Y.~V.},
      author={Le~Doussal, P.},
      author={Rosso, A.},
       title={Statistical mechanics of logarithmic {REM}: duality, freezing and
  extreme value statistics of {$1/f$} noises generated by {G}aussian free
  fields},
        date={2009},
        ISSN={1742-5468},
     journal={J. Stat. Mech. Theory Exp.},
      number={10},
       pages={P10005, 32},
      review={\MR{2882779}},
}

\bib{gnedenko}{article}{
      author={Gnedenko, B.},
       title={Sur la distribution limite du terme maximum d'une s\'erie
  al\'eatoire},
        date={1943},
        ISSN={0003-486X},
     journal={Ann. of Math. (2)},
      volume={44},
       pages={423\ndash 453},
      review={\MR{0008655 (5,41b)}},
}

\bib{gumbel}{book}{
      author={Gumbel, E.~J.},
       title={Statistics of extremes},
   publisher={Columbia University Press, New York},
        date={1958},
      review={\MR{0096342 (20 \#2826)}},
}

\bib{harper}{article}{
      author={Harper, A.~J.},
       title={A note on the maximum of the riemann zeta function, and
  log-correlated random variables},
        date={2013},
     journal={Preprint arXiv:1304.0677},
         url={http://arxiv.org/abs/1304.0677},
}

\bib{keating-snaith}{article}{
      author={Keating, J.~P.},
      author={Snaith, N.~C.},
       title={Random matrix theory and {$\zeta(1/2+it)$}},
        date={2000},
        ISSN={0010-3616},
     journal={Comm. Math. Phys.},
      volume={214},
      number={1},
       pages={57\ndash 89},
         url={http://dx.doi.org/10.1007/s002200000261},
      review={\MR{1794265 (2002c:11107)}},
}

\bib{Kistler2014}{incollection}{
      author={Kistler, N.},
       title={Derrida's random energy models},
        date={2015},
   booktitle={Correlated random systems: Five different methods},
      series={Lecture Notes in Math.},
      volume={2143},
   publisher={Springer, Berlin},
       pages={71\ndash 120},
}

\bib{lalley-sellke}{article}{
      author={Lalley, S.~P.},
      author={Sellke, T.},
       title={A conditional limit theorem for the frontier of a branching
  {B}rownian motion},
        date={1987},
        ISSN={0091-1798},
     journal={Ann. Probab.},
      volume={15},
      number={3},
       pages={1052\ndash 1061},
  url={http://links.jstor.org/sici?sici=0091-1798(198707)15:3<1052:ACLTFT>2.0.CO;2-0&origin=MSN},
      review={\MR{893913 (88h:60161)}},
}

\bib{Lawler2010}{book}{
      author={Lawler, G.~F.},
      author={Limic, V.},
       title={Random {W}alk: a {M}odern {I}ntroduction},
      series={Cambridge Studies in Advanced Mathematics},
   publisher={Cambridge University Press, Cambridge},
        date={2010},
      volume={123},
}

\bib{leadbetter}{book}{
      author={Leadbetter, M.~R.},
      author={Lindgren, G.},
      author={Rootz{\'e}n, H.},
       title={Extremes and related properties of random sequences and
  processes},
      series={Springer Series in Statistics},
   publisher={Springer-Verlag, New York-Berlin},
        date={1983},
        ISBN={0-387-90731-9},
      review={\MR{691492 (84h:60050)}},
}

\bib{madaule}{article}{
      author={Madaule, T.},
       title={Maximum of a log-correlated gaussian field},
        date={2014},
     journal={Preprint arXiv:1307.1365},
         url={http://arxiv.org/abs/1307.1365},
}

\bib{madaule_brw}{article}{
      author={Madaule, T.},
       title={Convergence in law for the branching random walk seen from its
  tip},
    language={English},
        date={2015},
        ISSN={0894-9840},
     journal={Journal of Theoretical Probability},
       pages={1\ndash 37},
         url={http://dx.doi.org/10.1007/s10959-015-0636-6},
}

\bib{MaillardZeitouni2015}{article}{
      author={Maillard, P.},
      author={Zeitouni, O.},
       title={Slowdown in branching brownian motion with inhomogeneous
  variance},
        date={2015},
     journal={Preprint arXiv:1307.3583},
}

\bib{montgomery-vaughan-multiplicative-nt}{book}{
      author={Montgomery, H.~L.},
      author={Vaughan, R.~C.},
       title={Multiplicative number theory. {I}. {C}lassical theory},
      series={Cambridge Studies in Advanced Mathematics},
   publisher={Cambridge University Press, Cambridge},
        date={2007},
      volume={97},
        ISBN={978-0-521-84903-6; 0-521-84903-9},
      review={\MR{2378655 (2009b:11001)}},
}

\bib{resnick}{book}{
      author={Resnick, S.~I.},
       title={Extreme values, regular variation and point processes},
      series={Springer Series in Operations Research and Financial
  Engineering},
   publisher={Springer, New York},
        date={2008},
        ISBN={978-0-387-75952-4},
        note={Reprint of the 1987 original},
      review={\MR{2364939 (2008h:60002)}},
}

\bib{rhodes-vargas_review}{article}{
      author={Rhodes, R.},
      author={Vargas, V.},
       title={Gaussian multiplicative chaos and applications: A review},
        date={2014},
     journal={Probab. Surveys},
      volume={11},
       pages={315\ndash 392},
         url={http://dx.doi.org/10.1214/13-PS218},
}

\bib{sound}{article}{
      author={Soundararajan, K.},
       title={Moments of the {R}iemann zeta function},
        date={2009},
        ISSN={0003-486X},
     journal={Ann. of Math. (2)},
      volume={170},
      number={2},
       pages={981\ndash 993},
         url={http://dx.doi.org/10.4007/annals.2009.170.981},
      review={\MR{2552116 (2010i:11132)}},
}

\bib{zeitouni_notes}{article}{
      author={Zeitouni, O.},
       title={Branching random walks and gaussian fields},
        date={2013},
     journal={Online lecture notes},
         url={\url{ http://www.wisdom.weizmasnn.ac.il/~zeitouni/
  notesGauss.pdf}},
        note={available at
  \url{http://www.wisdom.weizmann.ac.il/~zeitouni/notesGauss.pdf}},
}

\end{biblist}
\end{bibdiv}

\end{document}